\documentclass[12pt]{amsart}
\usepackage{amsmath,amssymb,amsbsy,amsfonts,latexsym,amsopn,amstext,
                                               amsxtra,euscript,amscd,bm,cite}
\usepackage{url}
\usepackage[colorlinks,linkcolor=blue,anchorcolor=blue,citecolor=blue,backref=page]{hyperref}
\usepackage{color}
\usepackage{float}
\usepackage{ mathrsfs }
\usepackage{comment}
\usepackage{enumitem}
\usepackage{bbm}

\hypersetup{breaklinks=true}

\usepackage[norefs,nocites]{refcheck}

\begin{document}

\newtheorem{theorem}{Theorem}
\newtheorem{lemma}[theorem]{Lemma}
\newtheorem{claim}[theorem]{Claim}
\newtheorem{cor}[theorem]{Corollary}
\newtheorem{prop}[theorem]{Proposition}
\newtheorem{definition}{Definition}
\newtheorem{question}[theorem]{Question}
\newtheorem{remark}[theorem]{Remark}
\newcommand{\hh}{{{\mathrm h}}}

\numberwithin{equation}{section}
\numberwithin{theorem}{section}
\numberwithin{table}{section}

\def\sssum{\mathop{\sum\!\sum\!\sum}}
\def\ssum{\mathop{\sum\ldots \sum}}
\def\iint{\mathop{\int\ldots \int}}

\def\squareforqed{\hbox{\rlap{$\sqcap$}$\sqcup$}}
\def\qed{\ifmmode\squareforqed\else{\unskip\nobreak\hfil
\penalty50\hskip1em\null\nobreak\hfil\squareforqed
\parfillskip=0pt\finalhyphendemerits=0\endgraf}\fi}

\newfont{\teneufm}{eufm10}
\newfont{\seveneufm}{eufm7}
\newfont{\fiveeufm}{eufm5}
%
%
\newfam\eufmfam
     \textfont\eufmfam=\teneufm
\scriptfont\eufmfam=\seveneufm
     \scriptscriptfont\eufmfam=\fiveeufm
%
%
\def\frak#1{{\fam\eufmfam\relax#1}}

\newcommand{\bflambda}{{\boldsymbol{\lambda}}}
\newcommand{\bfmu}{{\boldsymbol{\mu}}}
\newcommand{\bfxi}{{\boldsymbol{\xi}}}
\newcommand{\bfrho}{{\boldsymbol{\rho}}}

\newcommand{\bfalpha}{{\boldsymbol{\alpha}}}
\newcommand{\bfbeta}{{\boldsymbol{\beta}}}
\newcommand{\bfphi}{{\boldsymbol{\rho}}}
\newcommand{\bfpsi}{{\boldsymbol{\psi}}}
\newcommand{\bftheta}{{\boldsymbol{\vartheta}}}

\def\fK{Frak K}
\def\fT{Frak{T}}

\def\fA{{Frak A}}
\def\fB{{Frak B}}
\def\fC{\mathfrak{C}}

\def \balpha{\bm{\alpha}}
\def \bbeta{\bm{\beta}}
\def \bgamma{\bm{\gamma}}
\def \blambda{\bm{\lambda}}
\def \bchi{\bm{\chi}}
\def \bphi{\bm{\rho}}
\def \bpsi{\bm{\psi}}

\def\eqref#1{(\ref{#1})}

\def\vec#1{\mathbf{#1}}


\def\cA{{\mathcal A}}
\def\cB{{\mathcal B}}
\def\cC{{\mathcal C}}
\def\cD{{\mathcal D}}
\def\cE{{\mathcal E}}
\def\cF{{\mathcal F}}
\def\cG{{\mathcal G}}
\def\cH{{\mathcal H}}
\def\cI{{\mathcal I}}
\def\cJ{{\mathcal J}}
\def\cK{{\mathcal K}}
\def\cL{{\mathcal L}}
\def\cM{{\mathcal M}}
\def\cN{{\mathcal N}}
\def\cO{{\mathcal O}}
\def\cP{{\mathcal P}}
\def\cQ{{\mathcal Q}}
\def\cR{{\mathcal R}}
\def\cS{{\mathcal S}}
\def\cT{{\mathcal T}}
\def\cU{{\mathcal U}}
\def\cV{{\mathcal V}}
\def\cW{{\mathcal W}}
\def\cX{{\mathcal X}}
\def\cY{{\mathcal Y}}
\def\cZ{{\mathcal Z}}
\newcommand{\rmod}[1]{\: \text{mod} \: #1}

\def\cg{{\mathcal g}}

\def\vr{\mathbf r}

\def\e{{\mathbf{\,e}}}
\def\ep{{\mathbf{\,e}}_p}
\def\em{{\mathbf{\,e}}_m}

\def\Tr{{\mathrm{Tr}}}
\def\Nm{{\mathrm{Nm}\,}}

 \def\SS{{\mathbf{S}}}

\def\lcm{{\mathrm{lcm}}}
\def\ord{{\mathrm{ord}}}

\def\({\left(}
\def\){\right)}
\def\fl#1{\left\lfloor#1\right\rfloor}
\def\rf#1{\left\lceil#1\right\rceil}

\def\mand{\qquad \text{and} \qquad}

\newcommand{\commM}[1]{\marginpar{%
\begin{color}{red}
\vskip-\baselineskip 
\raggedright\footnotesize
\itshape\hrule \smallskip M: #1\par\smallskip\hrule\end{color}}}

\newcommand{\commI}[1]{\marginpar{%
\begin{color}{magenta}
\vskip-\baselineskip 
\raggedright\footnotesize
\itshape\hrule \smallskip I: #1\par\smallskip\hrule\end{color}}}

\newcommand{\commK}[1]{\marginpar{%
\begin{color}{blue}
\vskip-\baselineskip 
\raggedright\footnotesize
\itshape\hrule \smallskip K: #1\par\smallskip\hrule\end{color}}}




\hyphenation{re-pub-lished}

\mathsurround=1pt

\def\bfdefault{b}
\overfullrule=5pt

\def \F{{\mathbb F}}
\def \K{{\mathbb K}}
\def \N{{\mathbb N}}
\def \Z{{\mathbb Z}}
\def \Q{{\mathbb Q}}
\def \R{{\mathbb R}}
\def \C{{\mathbb C}}
\def\Fp{\F_p}
\def \fp{\mathfrak p}
\def \fq{\mathfrak q}

\def\ZK{\Z_K}

\def \xbar{\overline x}
\def\e{{\mathbf{\,e}}}
\def\ep{{\mathbf{\,e}}_p}
\def\eq{{\mathbf{\,e}}_q}


\title[Sum of a prime in arithmetic progression and square-free]{Representation of an integer as the sum of a prime in arithmetic progression and a square-free integer}

\date{\today}

\author[K. H. Yau]{Kam Hung Yau}

\address{Department of Pure Mathematics, University of New South Wales,
Sydney, NSW 2052, Australia}
\email{kamhungyau.math@gmail.com}

\begin{abstract}
Uniformly for small $q$ and $(a,q)=1$, we obtain an estimate for the weighted number of ways a sufficiently large integer can be represented as the sum of a prime congruent to $a$ modulo $q$ and a square-free integer.  Our method is based on the notion of local model developed by Ramar\'e and may be viewed as an abstract circle method.
\end{abstract}


\maketitle

\section{Introduction}

The Goldbach conjecture states that every even integer greater than two can be expressed as the sum of two primes. Although this remains an open problem due to the parity phenomenon, there are various progress and relaxations which contributes toward this direction.

Mordern sieve method can be traced back to the earlier works of Brun. In 1920, Brun~\cite{B1} showed that every sufficiently large even integer can be written as the sum of two numbers which have together at most nine prime divisors. Later, the celebrated result of Chen~\cite{C1} established that every sufficiently large even integer can be written as the sum of a prime and a number with at most two prime factors.  

Initiated  by Linnik~\cite{L2} in 1953, he showed that every sufficiently large even integer can be written as a sum of two primes and at most $K$ powers of two, where $K$ is an absolute constant although non-explicit. Many authors had made $K$ explicit where the best result $K=12$ is due to Liu \& L\"u~\cite{LL1}, improving the remarkable result $K=13$ by Heath-Brown \& Puchta~\cite{HBP1}.

Another relaxation is the ternary Goldbach conjecture which states that all odd integer greater than five is the sum  of three primes. Vinogradov~\cite{V1} developed a way to estimate sums over primes which combined with the circle method showed  the ternary Goldbach conjecture is true for all large odd integer greater than $C>0$. Recently, Helfgott~\cite{H1}  completed the proof of ternary Goldbach conjecture by sufficiently reducing the size $C$ and verified that no counterexample exists below $C$. 

We also have results when we replace one of the primes  in the Goldbach conjecture by a square-free integer. Estermann~\cite{E1} obtained an asymptotic formula for the number of representation of a sufficently large integer as the sum of a prime and a square-free number. Later, Page~\cite{P1} improved on the error term of Estermann~\cite{E1} and Mirsky~\cite{M1} improved and extended these results to count the number of representions of an integer as the sum of a prime and a $k$-free number. Recently, Dudek~\cite{D2} by tools of explicit number theory demonstrated that every integer greater than two can be a sum of a prime and a square-free integer.

In this paper we are motivated by a question posed in the PhD thesis of Dudek~\cite[Chapter 6, Problem 8]{D1}. Specifically Dudek asked for $(a,q)=1$, can all sufficiently large integer without local obstruction be a sum of a prime $p$ such that $p \equiv a \hspace{-1mm} \pmod{q}$ and a square-free integer. A similar problem on the number of representations of a sufficiently large integer to be the sum of two square-free integer has been studied by Evelyn \& Linfoot~\cite{EL} and later simplified by Estermann~\cite{E2}. The best result is achieved by Br\"udern \& Perelli~\cite{BP}.

There are mainly two advanced techniques for attacking certain binary additive problems: sieve method~\cite{FI1} and the dispersion method of Linnik~\cite{L1}.
We note that the circle method~\cite{V2} has certain difficulity in dealing with binary additive problems.
Our method applied here is due to Ramar\'e~\cite{R1} on his notion of local model, and can be viewed as an abstract circle method. We remark that Heath-Brown~\cite{HB1} had already notice this connection for his alternate prove of Vinogradov's three prime theorem~\cite{V1}.

Lastly the techniques used here may be adapted for various other binary additive problems. In particular, the author expects that it should be possible to prove an asymptotic bound for the number of representation of an integer as the sum of a square-free integer and a prime $p$ such that $p+1$ is square-free.

\newpage

\section{Notation}

The statements $U=O(V)$ and $U \ll V$ are both equivalent to $|U| \le c|V|$ for some fixed positive constant $c$. If $c$ depends on a parameter, say $a$, then we may write $U=O_a(V)$ or $U \ll_a V$.

For completeness, we recall the following standard notation in analytic number theory.
\begin{align*}
a \mid b & : \mbox{$a$ divides $b$.} \\
a \nmid b  & : \mbox{$a$ does not divide $b$.} \\
p & : \mbox{with or without subscript is exclusively a prime number.} \\
p^{\alpha} \mid\!\mid n & : \mbox{means $p^{\alpha} \mid n$ but $p^{\alpha+1} \nmid n$.}  \\
(a_1, \ldots, a_n) & : \mbox{the greatest common divisor of $a_1, \ldots, a_n$.} \\
[b_1, \ldots, b_n] & : \mbox{the lowest common multiple of $b_1, \ldots, b_n$.}  \\
\sigma(n) & : \mbox{the sum of all positive divisors of $n$.} \\
\varphi(n) & : \mbox{the number of integers in $[1,n]$ coprime to $n$.} \\
\mu & : \mbox{the M\"{o}bius function.} \\
\Lambda & : \mbox{the von Mangoldt function.} \\
a  \equiv b [k] &: \mbox{means $a \equiv b \hspace{-3mm} \pmod k$.} \\
\mathbf{e}_r(t) & : \mbox{is $\exp(2 \pi i t/r)$.} \\
\displaystyle \sideset{}{^*} \sum_{a=1}^{n} & : \mbox{a summation over the integers in $[1,n]$ coprime to $n$.} \\
c_r(k) &: \mbox{the Ramanujan sum $\displaystyle \sideset{}{^*} \sum_{\substack{a=1} }^{r} \mathbf{e}_r(ka)$.} \\
\mathbbm{1}_{S} & : \mbox{is $1$ if $S$ is true and $0$ otherwise.} \\
\mathcal{L} & : \mbox{is $\log N$.}
\end{align*}

\newpage

\section{Main result}
We denote
$$
\mathcal{R}_{a,q}(N) = \sum_{ \substack{N = p+n \\ p \equiv a [q] }} \mu^2(n)  \log p 
$$
to be the weighted number of representations for $N$ as the sum of a  prime congruent to $a$ modulo $q$ and a square-free integer.

We now state a bound for $\mathcal{R}_{a,q}(N)$ which is uniform for small $q$.

\begin{theorem} \label{T1}
Let $C_1,C_2>0$ then we have
\begin{equation} \label{eqn: thm eqn}
\mathcal{R}_{a,q}(N) = \mathfrak{S}_{a,q}(N) N \left  \{ 1 + O_{C_1,C_2}  \left (\mathcal{L}^{-C_2} \right  )   \right \}
\end{equation}
uniformly for $(a,q)=1$ and $q \le \mathcal{L}^{C_1}$. The \textit{singular series} is given by
\begin{align*}
 \mathfrak{S}_{a,q}(N) & = \frac{6}{\varphi(q) \pi^2} \prod_{\substack{p=2 \\ (p,q)=1 }}^{\infty}  \left (1+ \frac{c_p(N)}{(p^2-1)(p-1)} \right ) \prod_{\substack{p \mid\!\mid q }}  \left (1 - \frac{c_p(N-a)}{p^2-1} \right )  \\
&\quad  \times \prod_{p^2 \mid q} \left  (1 - \frac{c_p(N-a) + c_{p^2}(N-a) }{p^2-1} \right ).
\end{align*}
\end{theorem}

The implied constant in the reminder term is ineffective as the Siegel-Walfisz Theorem is ineffective.
In view of Lemma~\ref{lem: ramanujan sum property}, we can rewrite the \textit{singular series} in a more rudimentary form
\begin{align*}
& \mathfrak{S}_{a,q}(N) \\
&  =  \mathbbm{1}_{\substack{\forall p, p^2 \nmid (q ,N-a)}}  \cdot \frac{    6}{\varphi(q) \pi^2} \prod_{\substack{  p \mid N \\ p \nmid q  }} \left (1 + \frac{1}{p^2-1} \right ) \prod_{\substack{p=2 \\ p\nmid Nq  }}^{\infty}  \left (1 - \frac{1}{(p^2-1)(p-1)} \right ) \\
& \quad \times  \prod_{\substack{p \mid\!\mid q \\ p \mid (N-a) }} \left (1 - \frac{1}{p+1} \right ) \prod_{\substack{p \mid\!\mid q \\ p \nmid (N-a) }} \left (1 + \frac{1}{p^2-1} \right )  \prod_{\substack{p^2|q  \\ p \nmid (N-a) }} \left ( 1 + \frac{1}{p^2-1}\right ) \\
& \quad \times \prod_{\substack{p^2|q \\ p \mid\!\mid (N-a) }} \left (1 + \frac{1}{p^2-1} \right ).
\end{align*}

If $p_1^2 \mid q$ and $p_1^2 \mid (N-a)$ then it follows for all primes $p \equiv a [q]$ that $N-p$ is never square-free and hence $\mathcal{R}_{a,q}(N)=0$. This coincide exactly to the case when $\mathfrak{S}_{a,q}(N)$ vanishes.

Using the well-known convolution identity $\mu^2(n) = \sum_{d^2|n} \mu(d)$ and assuming the GRH (Generalised Riemann Hypothesis), the error term in ~\eqref{eqn: thm eqn} can be replaced by $O_{\varepsilon}(N^{3/4+\varepsilon})$. Indeed on the GRH, we obtain
$$
N \sum_{\substack{a = 1 \\ (a,N)=1}}^{\infty} \frac{\mu(a)}{\varphi([a^2,q])} \cdot \mathbbm{1}_{(a^2,q) \mid (N-a)} + O_{\varepsilon}  \left (N \sum_{a > w} \frac{1}{a^2}  + N^{1/2+\varepsilon} w\right ).
$$
Taking $w = N^{1/4}$ gives the result. It seems the error term cannot be further improved (assuming the GRH) using the local model method.

Observe that when we take $q=1$ in Theorem~\ref{T1}, we obtain a special case of Mirsky~\cite{M1} (after weighing but with a weaker error term). Indeed our \textit{singular series} simplifies to
\begin{align*}
\mathfrak{S}_{0,1}(N) &  = \prod_{p=2}^{\infty} \left  (\frac{p^2-1}{p^2} \right ) \prod_{\substack{p=2 \\ p \mid N}}^{\infty} \left  (1 +\frac{1}{p^2-1} \right ) \prod_{\substack{p=2 \\ p \nmid N}}^{\infty} \left (1 -\frac{1}{(p^2-1)(p-1)} \right ) \\
& = \prod_{\substack{p=2 \\ p \nmid N}}^{\infty}  \left ( \frac{p^2-1}{p^2} - \frac{1}{p^2 (p-1)}   \right )\\
&  = \prod_{\substack{p=2 \\ p \nmid N}}^{\infty} \left  ( 1 - \frac{1}{p(p-1)} \right ).
\end{align*}

\section{Outline of Method}
For a more thorough exposition, see~\cite[Chapter 1,4 \&17]{R1}. Our method will model that of~\cite[Chapter 19]{R1} where therein Ramar\'e proved an asymptotic bound for the number of representations of a sufficiently large integer as a sum of two square-free integer.

We press forward and recall the definition of \textit{local} and \textit{global} product, see~\eqref{eq: local product} and~\eqref{eq: global product} respectively.

Set $\mathcal{H}$ as an \textit{almost orthogonal system} by the following collection of information:
\begin{enumerate}[label=(\roman*)]
\item a finite family $(\phi_i^*)_{i \in I}$ of elements of $\mathcal{H}$, 
\item a finite family $(M_i)_{i \in I}$ of positive real numbers,
\end{enumerate}
and
$$
\left  \lVert \sum_{i \in I} \xi_i \phi_i^*  \right  \rVert^2 \le   \sum_{i \in I }  M_i |\xi_i|^2 
$$
for all $(\xi_i)_{i \in I} \in \mathbb{C}^I$.


The special case of choosing an orthogonal basis is enlightening. If $(\phi_i^*)_{i \in I}$ were orthogonal then we may take $M_i = \lVert \phi_i^* \rVert^2$.


Let $\mathscr{F}(\mathbb{N})$ be the vector space of all complex valued functions on the positive integers and
\begin{equation} \label{eq: global product}
[f|g] = \sum_{n  \le N} f(n) \overline{g(n)}
\end{equation} 
 be the usual scalar product (\textit{global} Hermitian product) for all $f,g \in \mathscr{F}(\mathbb{N})$.

Let $(a,q)=1$, $f(n)= \Lambda(n) \mathbbm{1}_{n \equiv a [q]}$ and $g(n) =  \mu^2(N-n)$. Observe that 
\begin{align*}
[f|g] &= \sum_{n \le N } \Lambda(n)  \mathbbm{1}_{n \equiv a [q]} \cdot  \mu^2(N-n)\\
&= \sum_{ \substack{N = n_1 + n_2\\ n_1 \equiv a [q]}} \Lambda(n_1)  \mu^2(n_2)
\end{align*}
is the weighted number of representations for $N$ to be a sum of a prime congruent to $a$ modulo $q$ and a square-free integer. Our ultimate goal is to compute $[f|g]$ and we shall use the notion of local model to this end.

Indeed let $\mathcal{Q} \subseteq \mathbb{N}$ be a carefully chosen set of moduli, we construct two local models $( \rho_q^* )_{q \in \mathcal{Q}}$ and $( \gamma_q^* )_{q \in \mathcal{Q}}$ to approximate $f$ and $g$ respectively, and in some sense they are made to copy the distribution of $f$ and $g$ in arithmetic progression respectively. 

Next we take $(\frac{1}{2} \Delta_q^* )_{q \in \mathcal{Q}}$ to be essentially the union of some linear combination of $(  \rho_q^* )_{q \in \mathcal{Q}}$ and $(   \gamma_q^* )_{q \in \mathcal{Q}}$, this will be the local model accountable for both $f$ and $g$. Furthermore for all $q_1, q_2 \in \mathcal{Q}$, we set $M_{q_1} = \sum_{t \in \mathcal{Q} } |[\Delta_{t}^* | \Delta_{q_1}^*]|$, see~\cite[Lemma 1.1]{R1}. This gives rise to an \textit{almost orthogonal system} and in particular will imply the scalar product
$$
\left [f  - \frac{1}{2}   \sum_{q \in \mathcal{Q} }  \xi_q(f)  \Delta_q^*  \Bigg |g - \frac{1}{2} \sum_{q \in \mathcal{Q} } \xi_q(g) \Delta_q^*  \right ]
$$
is small in a suitable sense. The construction will also secure $[\Delta_{q_1}^*|\Delta_{q_2}^*]$ to be small when $q_1 \neq q_2$.
Expanding the inner product, we have
\begin{align*}
 \Bigg [f  - \frac{1}{2}   \sum_{q \in \mathcal{Q} }  \xi_q(f) \Delta_q^*  \Bigg |  g  - &  \frac{1}{2} \sum_{q \in \mathcal{Q} } \xi_q(g) \Delta_q^*  \Bigg ]\\
& = [f|g] - \frac{1}{2}\sum_{q \in \mathcal{Q}} \xi_q(f) [\Delta_q^*|g] - \frac{1}{2} \sum_{q \in \mathcal{Q}} \xi_q(g)[f|\Delta_q^*] \\
& \quad  + \frac{1}{4} \sum_{q_1,q_2 \in \mathcal{Q} } \xi_{q_1}(f) \xi_{q_2}(g) [\Delta^*_{q_1}|\Delta^*_{q_2}].
\end{align*}

Here we take $\xi_q(f) = [f|\Delta_q^*]/M_q$ and $\xi_q(g) = [g|\Delta_q^*]/M_q$ as motivated by the orthogonal case. Simplifying gives
$$
[f|g] = \sum_{q \in \mathcal{Q} }  \frac{ [f|\Delta_q^*] [\Delta_q^*|g]}{M_q} +O(R).
$$
The error term can be shown to be sufficiently small by appealing to the local model. The summand in the sum above can then be replaced by a tractable expression for which we can compute explicitly and the result soon follows.

\section{Preparations}

\subsection{Number theoretical considerations}
We record here various number theoretical lemmas needed in subsequent sections. For completeness we will also include several straightforward lemmas that may be applied freely without reference.

First we recall the well-known orthogonality of exponential sums~\cite[Equation 4.1]{MV1}.
\begin{lemma}
For any positive integer $k$, we have
$$
\sum_{a=1}^{k}  \mathbf{e}_k(ar) =
\begin{cases}
k &\mbox{if $r \equiv 0 [k]$,} \\
0 & \mbox{otherwise.}
\end{cases}
$$
\end{lemma}

Next we recall the Chinese remainder theorem for arbitrary modulus~\cite[Theorem 3.12]{JJ1}.
\begin{lemma}[Chinese remainder theorem]
Let $a_i, m_i \in \mathbb{N}$ for $i=1, \ldots, n$ and $L=[m_1, \ldots, m_n]$. The following system of congruences $x \equiv a_i [m_i]$ for $i=1, \ldots, n$ is solvable if and only if $a_i \equiv a_j [(m_i,m_j)]$ for any $1 \le i,j \le n$. If the system is solvable then $x \equiv \sigma [L]$ for some $1 \le \sigma \le L $and any two such $x$ are congruent modulo $L$. Moreover, $(a_i,m_i)=1$ for $1 \le i \le n$ if and only if $(\sigma,L)=1$.
\end{lemma}

We recall from~\cite[Theorem 4.1]{MV1} some fundamental properties of Ramanujan sums.

\begin{lemma} \label{lem: ramanujan sum property}
For any positive integers $n$ and $r$, the Ramanujan sum $c_r(n)$ is a multiplicative function of $r$. Moreover we have
$$
c_r(n) = \sum_{d \mid (r,n)} d \mu \left (\frac{r}{d} \right ) = \frac{\mu(r/(r,n))}{\varphi(r/(r,n))} \varphi(r),
$$
and hence $|c_r(n)| = \varphi((r,n))$. In particular
$$
c_p(n) =
\begin{cases}
-1 & \mbox{if $p \nmid n$,} \\
\varphi(p)  & \mbox{if $p\mid n$,}
\end{cases}
\mbox{\hspace{2mm} and \hspace{2mm}  }
c_{p^2}(n) = 
\begin{cases}
0 & \mbox{if $p \nmid n$,} \\
-p & \mbox{if $p \mid\!\mid  n$,} \\
\varphi(p^2) & \mbox{if $p^2 \mid n$.}
\end{cases}
$$

\end{lemma}

The next result provides an explicit expression for detecting equality for divisors~\cite[Corollary 3.1]{R1}.

\begin{lemma} \label{lem: if d=a}
For integer $a \ge 1$ and any divisor $d$ of $a$, we have
$$
 \sum_{\substack{k \mid a \\ d \mid k}} \mu(k/d) = \mathbbm{1}_{d=a}. 
$$
\end{lemma}

We record below a sensational result which gives an estimate for the number of primes in an arithmetic progression for small moduli~\cite[Corollary 5.29]{IK1}.

\begin{lemma}[Siegel-Walfisz] \label{lem: siegel-walfisz}
For any $A,B>0$, we have
$$
\sum_{\substack{n \le N \\ n \equiv a [q] }} \Lambda(n) = \frac{N}{\varphi(q)} +O_{A,B }(N\mathcal{L}^{-B})
$$
uniformly for $(a,q)=1$ and $q \le \mathcal{L}^A$.
\end{lemma}

Finally we state an auxiliary lemma that we will need later.

\begin{lemma} \label{lem: double mobuis sum}
For all cubefree positive integers $m$ and $n$, we have
$$
 \sum_{\substack{d_1 \mid m }} \sum_{\substack{d_2 \mid n}}   \mu(m/d_1 )  \mu(n/d_2 )(d_1,d_2) =  \varphi(m) \cdot  \mathbbm{1}_{m=n}.
$$
\end{lemma}

\begin{proof}
Write $n= n_1 n_2^2$ and $m=m_1 m_2^2$ where $\mu^2(n_1 n_2)= \mu^2(m_1 m_2)=1$. We factor our sum
\begin{align*}
& \mu^2(n_1) \mu^2(n_2) \sum_{d_1 \mid m}   \mu \left (\frac{m}{d_1} \right )  \sum_{d_2 \mid n_1} \mu \left( \frac{n_1}{d_2} \right )(d_1, d_2) \sum_{d_2' \mid n_2^2} \mu \left (\frac{n_2^2}{d_2'} \right ) (d_1, d_2') \\
& =  \sum_{d_1 \mid m}  \mu \left (\frac{m}{d_1} \right )   \prod_{p \mid n_1}  ((d_1,p)-1   ) \prod_{p \mid n_2} ((d_1,p^2) - (d_1,p) ).
\end{align*}
If there exist a prime factor that divides $n$ but not $m$ then the sum vanishes. By symmetry, we are left to consider the case $m=n$ and in turn the sum simplifies to
\begin{align*}
\prod_{p \mid m_1}  (  (m,p) -1 )   \prod_{p \mid m_2}(  (m, p^2) -(m,p) ) 
&  =  \prod_{p \mid m_1}(p-1) \prod_{p \mid m_2} p(p-1) \\
& = \varphi(m_1) \varphi(m_2^2) \\
& = \varphi(m).
\end{align*}
\end{proof}

\subsection{Arithmetic functions in arithmetic progression}

The following result provides an asymptotic for the number of square-free integers in an arithmetic progression. 

Note that we canonically extend the characteristics function of the square-free integers to negative integers by $\mu^2(-n)=\mu^{2}( |n| )$. We recall a result from~\cite[Lemma 19.1]{R1}.

\begin{lemma} \label{lem: sqfr in AP bound}
For any $\varepsilon >0$, we have
$$
\sum_{ \substack{ n \le N \\ n \equiv a [q]  }} \mu^2(N-n) = (N/q) \gamma_q(N-a) + O_{\varepsilon} \left (q^{\varepsilon} \sqrt{N/q} \right ).
$$
 Here
$$
\gamma_q(a) =
\begin{cases}
\displaystyle 0 & \mbox{if there exists $p^2 \mid (a,q)$,} \\[8pt]
\displaystyle \frac{6}{\pi^2} \prod_{p \mid q} \frac{p^2}{p^2-1} \prod_{\substack{p \mid\!\mid q \\ a \equiv 0 [p]}} \Big (1-\frac{1}{p} \Big ) & \mbox{otherwise.}
\end{cases}
$$
\end{lemma}

As in the notation of~\cite[Chapter 19]{R1}, for cubefree $q$ we define
$$
\gamma_q^*(a) = \sum_{d \mid q} \mu  \left (  \frac{q}{d} \right ) \gamma_d(a)
$$
with $a \in \mathbb{Z}/q\mathbb{Z}$ and
\begin{equation} \label{eq: t}
t(q)= \prod_{p \mid q} \frac{-1}{p^2-1}.
\end{equation}

Note that in some sense $\gamma_q^*(N-n)$ is defined to imitate $\mu^2(N-n)$ in arithmetic progression.

We recall the following result from~\cite[Lemma 19.2]{R1} which provides an explicit expression for $\gamma_q^*$.
\begin{lemma} \label{lem: explicit gamma}
We have $\gamma_q^*(a) = 6t(q)c_q(a)/\pi^2$ if $q$ is a positive cubefree integer, while if $q$ has a cubic factor greater than $1$ then $\gamma_q^*(a)=0$.
\end{lemma}

Let $(a',q')=1$ and $q$ be a  positive cubefree integer, denote
\begin{equation} \label{def:rho_q}
\rho_q(a)=
\begin{cases}
0 & \mbox{if $(a,q)>1$ or $(q,q') \nmid (a-a')$,} \\[3pt]
\displaystyle \frac{q}{\varphi([q,q'])} & \mbox{if $(q,q') \mid (a-a')$ and $(a,q)=1$.}
\end{cases}
\end{equation}

To motivate the definition of $\rho_q$,  consider the  sum
$$
S= \sum_{\substack{n \le N \\ n \equiv a' [q'] \\  n \equiv a [q]}} \Lambda(n).
$$

By the Chinese remainder theorem, the simultaneous congruence equations is solvable if and only if $(a,q) =1$ and $(q,q') \mid (a-a')$. If this is the case then for some $0\le \sigma < [q,q']$ we should expect
$$
S = \sum_{\substack{n \le N \\  n \equiv \sigma [ [q,q'] ]  }} \Lambda(n)  = \rho_q(a)  \frac{N}{q} + o \left (\frac{N}{\varphi([q,q'])} \right ).
$$

We remark that if $a_1 \equiv a_2 [q]$ then $\rho_{q}(a_1) = \rho_{q}(a_2)$.

\begin{lemma} \label{lem:rho multi}
For $q_1, q_2$ positive cubefree integers with $(q_1,q_2)=1$, we have
$$
 \rho_{q_1 q_2}(a) = \varphi(q') \rho_{q_1 }(a)  \rho_{ q_2}(a).
$$
In particular for any positive integer $a$, $\varphi(q')\rho_{q}(a)$
is a multiplicative function of $q$.
\end{lemma}

\begin{proof}
If $(a,q_1 q_2)>1$ then either $(a,q_1)>1$ or $(a,q_2)>1$. If $(q_1 q_2,q') \nmid (a-a')$ then we have  $(q_1,q') \nmid (a-a')$ or $(q_2,q') \nmid (a-a')$. Consequently for both cases we obtain
$$
 \rho_{q_1 q_2}(a)=\rho_{q_1 }(a)  \rho_{ q_2}(a)=0.
$$

Clearly $(q_1 q_2,q') \mid (a-a')$ and $(a,q_1 q_2)=1$ if and only if we satisfy both the conditions
$$
\begin{cases}
\mbox{$(q_1,q') \mid (a-a')$ and $(a,q_1)=1$,} \\
\mbox{$(q_2,q') \mid (a-a')$ and $(a,q_2)=1$.}
\end{cases}
$$
Hence it is enough to show
$$
\varphi(q')\varphi([q_1 q_2, q'])  = \varphi([q_1,q']) \varphi ([q_2,q']).
$$

Let us recall the identity
$
ab = [a,b] \cdot (a,b)
$
for all positive integers $a,b$. Consequently
\begin{align*}
\varphi(q')\varphi([q_1 q_2, q']) &  = \varphi(q')  \varphi \left (\frac{q_1 q_2 q'}{(q_1,q') (q_2,q')} \right ) \\
& =\varphi(q')  \varphi \left (  \frac{[q_1,q'] \cdot [q_2,q']}{q'} \right)\\
& =\varphi([q_1,q']) \varphi ([q_2,q']),
\end{align*}
since 
$$
\left ([q_1,q'], \frac{[q_2,q']}{q'} \right )= \left (q',\frac{[q_2,q']}{q'} \right )=1.
$$
The result follows immediately.
\end{proof}

We are now ready to define our local approximation for $f$. Set
$$
\rho_q^*(c) = \sum_{d|q} \mu \left ( \frac{q}{d}  \right ) \rho_d(c)
$$
for any positive cubefree integer $q$.

The next result provides an explicit expression for $\rho_q^*$.
\begin{lemma} \label{lem: rho_q^* properties}
Let $q=q_1q_2^2$ with $\mu^2(q_1q_2)=1$. If $a_1 \equiv a_2 [q]$ then $\rho_{q}^*(a_1) = \rho_{q}^* (a_2)$.
 For any positive integer $a$, $ \varphi(q')\rho_q^*(a)$ is a multiplicative function of $q$. Furthermore  
\begin{align*}
\rho_q^*(a) & =  \mathbbm{1}_{q_2^2 \mid q'} \cdot \frac{   \mu(q_1/(q_1,q'))  c_{q_1/(q_1,q')}(a)  c_{(q_1,q')q_2^2}(a-a')}{  \varphi(q') \varphi(q_1/(q_1,q'))}.
\end{align*}
\end{lemma}

\begin{proof}

We note that if $a_1 \equiv a_2 [q]$ then $a_1 \equiv q_2 [d]$ for all divisors $d$ of $q$. Hence
\begin{align*}
\rho_q^*(a_1)  =  \sum_{d|q} \mu(q/d) \rho_d(a_1)  = \sum_{d|q} \mu(q/d) \rho_d(a_2) =\rho_{q}^*(a_2).
\end{align*}

Appealing to Lemma~\ref{lem:rho multi}, $ \varphi(q') \rho_{q}^*$ is a Dirichlet convolution of two multiplicative functions, hence $ \varphi(q') \rho_{q}^*$ is multiplicative in $q$. Therefore we factor
$$
\varphi(q') \rho_{q}^* = \varphi(q') \rho_{q_1}^* \cdot    \varphi(q') \rho_{q_2^2}^*.
$$
We first consider $\varphi(q') \rho_{q_1}^*$. Recall that $\rho_1(a) =1/\varphi(q')$ and so
\begin{align*}
 \varphi(q') \rho_{p}^*(a) &=  \varphi(q') \rho_p(a)  -1 .
\end{align*}
From~\eqref{def:rho_q}, we check 
\begin{equation} \label{rho_p^* explicit expression}
 \varphi(q') \rho_{p}^*(a)=
\begin{cases}
-1 & \mbox{if $p \nmid q'$, $p \mid a$,} \\
1/(p-1) & \mbox{if $p \nmid q'$, $p \nmid a$,} \\
-1 & \mbox{if $p \mid q'$, $p \nmid (a-a')$,} \\
p-1  & \mbox{if $p\mid q'$, $p\mid (a-a')$.}
\end{cases}
\end{equation}
In the last line, we note that the condition implies that $(a,p)=1$ since $(a',q')=1$. Therefore
\begin{align*}
\varphi(q') \rho_{q_1}^*(a) & = \prod_{\substack{ p \mid q_1 \\ p \nmid q' \\  p \mid a}}(-1) \prod_{\substack{p \mid q_1 \\ p \nmid q' \\p \nmid a }} \frac{1}{p-1} \prod_{ \substack{p \mid (q_1,q') \\p \nmid (a-a')}} (-1) \prod_{\substack{p \mid (q_1,q') \\ p \mid (a-a')}}( p-1) \\
& = \frac{\mu(q_1/(q_1,q'))c_{q_1/(q_1,q')}(a) c_{(q_1,q')}(a-a')}{\varphi(q_1/(q_1,q'))}.
\end{align*}

Lastly, we turn to $ \varphi(q') \rho_{q_2^2}^*$ and we see  for any $p^2 \mid q_2^2$ that
\begin{align*}
 \varphi(q') \rho_{p^2}^* (a) & = \sum_{d \mid p^2} \mu(p^2/d) \varphi(q') \rho_{d}(a) \\
 &= \varphi(q') \rho_{p^2}(a)-\varphi(q')\rho_{p}(a).
\end{align*}
We readily check that $\varphi(q') \rho_{p^2}^*(a)$ is
$$
\begin{cases}
0 &\mbox{if $(p^2,q')=1$, $p \nmid a$,} \\
0 &\mbox{if $(p^2,q')=p$, $p \nmid a$, $p \mid (a-a')$,} \\
0 &\mbox{if $(p^2,q')=p$, $p \nmid a$,  $p\nmid (a-a')$,} \\
0  &\mbox{if $(p^2,q')=p^2$, $p \nmid a$,  $p \nmid (a-a')$,} \\
-p &\mbox{if $(p^2,q')=p^2$, $p \nmid a$, $p \mid \! \mid (a-a')$,} \\
 p(p-1) &\mbox{if $(p^2,q')=p^2$, $p \nmid a$,   $p^2 \mid (a-a')$,} \\
0 & \mbox{if $p \mid a$.} \\
\end{cases} 
$$
We condense the expression to the simpler 
$$
 \varphi(q') \rho_{p^2}^*(a) =
\begin{cases}
0 &\mbox{if $(p^2,q') \le p$,} \\
 0  &\mbox{if $p^2 \mid q'$,  $p \nmid (a-a')$,} \\
 -p &\mbox{if $p^2 \mid q'$, $p \mid \! \mid (a-a')$, } \\
 p(p-1) &\mbox{if $p^2 \mid q'$,   $p^2 \mid (a-a')$}, 
\end{cases} 
$$
since $(a',q')=1$.
Hence we obtain
$$
 \varphi(q') \rho_{q_2^2}^*(a)= \mathbbm{1}_{q_2^2|q'} \cdot  c_{q_2^2}(a-a').
$$
The result follows immediately.
\end{proof}

\subsection{Local Hermitian product}

In this section we compute various \textit{local} Hermitian products explicitly. 

We denote \begin{equation} \label{def: theta_q^*}
\vartheta_q^*(a) =  \gamma_q^*(N-a)
\end{equation}
for all positive cubefree integer $q$.

For fixed $q$, we denote  $\mathscr{F}( \mathbb{Z}/q \mathbb{Z})$ to be the vector space of complex valued functions over $\mathbb{Z}/q \mathbb{Z}$. We endow this vector space with the \textit{local} Hermitian product by setting
\begin{equation} \label{eq: local product}
[f|g]_q = \frac{1}{q}  \sum_{n \bmod q} f(n) \overline{g(n)}
\end{equation}
for all $f,g \in \mathscr{F}(\mathbb{Z}/q\mathbb{Z})$.

We now state an explicit expression for the norms of $\vartheta_q^*$ and $\rho_q^*$.

\begin{lemma} \label{lem: theta,rho norms}
For all $q =q_1q_2^2$ with $\mu^2(q_1 q_2)=1$, we have
$$
\lVert \vartheta_q^* \rVert_q^2 =  \left (\frac{6 t(q)}{\pi^2} \right )^2  \varphi(q)
$$
 and
$$
\lVert \rho_q^* \rVert_q^2 = \mathbbm{1}_{q_2^2|q'} \cdot  \frac{ \varphi(q_2^2)\varphi((q_1,q'))^2  }{ \varphi(q')^2 \varphi(q_1)}.
$$
\end{lemma}

\begin{proof}
The expression for $\lVert \vartheta_q^* \rVert_q^2$ can be derived as in~\cite[Equation (19.10)]{R1}.

Write
\begin{align*}
\lVert \rho_q^* \rVert_q^2  & = \frac{1}{q \varphi(q')^2} \sum_{a \bmod q} \left  |  \varphi(q') \rho_q^*(a) \right |^2 \\
&= \frac{1}{  q \varphi(q')^2}  S(q), \mbox{ say}.
\end{align*}
If $q_2^2 \nmid q'$ then we are done since $\rho_q^*=0$ by Lemma~\ref{lem: rho_q^* properties}. Otherwise, by the Chinese remainder theorem we factor
\begin{align*}
S(q) & =  S(q_1) S(q_2^2) \\
& = \prod_{p|q_1} S(p) \prod_{p^2|q_2^2} S(p^2).
\end{align*}

Appealing to~\eqref{rho_p^* explicit expression}, we readily check
$$
S(p) = 
\begin{cases}
1+1/(p-1) =p/(p-1)& \mbox{if $p \nmid q'$,} \\
(p-1)^2 +(p-1)= p(p-1) & \mbox{if $p \mid q'$.}
\end{cases}
$$

By Lemma~\ref{lem: rho_q^* properties}, expanding the Ramanujan sum and interchanging the summation, we obtain
\begin{align*}
S(p^2)&  =  \sum_{a \bmod p^2} \left |c_{p^2}(a-a')  \right |^2 \\
&=  \sideset{}{^*} \sum_{ \substack{r_1 =1} }^{p^2}      \sideset{}{^*} \sum_{\substack{r_2 =1}}^{p^2}  \mathbf{e}_{p^2} \left  ( a'(r_2 -r_1) \right )   \sum_{a \bmod p^2}  \mathbf{e}_{p^2} \left ( a(r_1 -r_2) \right ).
\end{align*}
By orthogonality, we get
$$
S(p^2) = p^2 \varphi(p^2).
$$

It follows 
\begin{align*}
\lVert \rho_q^* \rVert_q^2 
& = \frac{1}{q \varphi(q')^2}  \prod_{ \substack{p|q_1 \\ p \nmid q'}} \frac{p}{p-1} \prod_{\substack{p|q_1 \\ p|q'}} p (p-1) \prod_{p|q_2} (p^2 \varphi(p^2)) \\
& = \frac{1}{\varphi(q')^2} \prod_{ \substack{p|q_1 \\ p \nmid q'}} \frac{1}{p-1} \prod_{\substack{p|q_1 \\ p|q'}}  (p-1) \prod_{p|q_2} \varphi(p^2) \\
& = \frac{\varphi(q_2^2) \varphi((q_1,q'))^2}{\varphi(q')^2 \varphi(q_1)}, 
\end{align*}
since 
$$
\left ((q_1,q'), \frac{q_1}{(q_1,q')} \right )=1.
$$
\end{proof}

For all positive cubefree integer $q$, we denote
$$
b(q) =  \sideset{}{^*} \sum_{r=1}^{q}  \mathbf{e}_q \left ( rN \right ) h_q(r),
$$
where
$$
h_q(r)=  \sum_{a \bmod q} \varphi(q') \rho^*_q(a)  \mathbf{e}_q \left ( -ra \right ).
$$

We now state a result which provides an explicit  expression for $b(p)$ and $b(p^2)$.

\begin{lemma} \label{b explicit expression}
The function $b(q)$ is a multiplicative function of $q$. Moreover we have
$$
b(p) = 
\begin{cases}
\displaystyle \frac{-p c_p(N)}{\varphi(p)} & \mbox{if $p \nmid q'$,} \\[8pt]
\displaystyle p c_{p}(N-a') & \mbox{if $p \mid q'$,}
\end{cases}
$$
and
$$
b(p^2) = \mathbbm{1}_{p^2 |q'} \cdot p^2 c_{p^2}(N-a').
$$
\end{lemma}

\begin{proof}
For positive cubefree integers $q_1,q_2$ with $\mu^2(q_1 q_2)=1$, we have by the Chinese remainder theorem 
\begin{align*}
h_{q_1 q_2} (r) & = \sum_{\substack{a_1 \bmod q_1 \\ a_2 \bmod q_2}}  \varphi(q')  \rho^*_{q_1 q_2} ( a_1 q_2 + a_2 q_1)  \mathbf{e}_{q_1} \left ( -ra_1 \right )   \mathbf{e}_{q_2} \left ( -ra_2 \right ) \\
& = \sum_{\substack{a_1 \bmod q_1 \\ a_2 \bmod q_2}}   \varphi(q')  \rho^*_{q_1 } ( a_1 )   \varphi(q') \rho^*_{q_2 } ( a_2 )   \mathbf{e}_{q_1} \left ( -ra_1 \right )   \mathbf{e}_{q_2} \left ( -ra_2 \right ) \\
& = h_{q_1} (r) h_{ q_2} (r),
\end{align*}
where in the second line we used Lemma~\ref{lem: rho_q^* properties}. Similarly, we show that $b(q)$ is a multiplicative function in $q$.

First we assume that $(p,q')=1$, appealing to~\eqref{rho_p^* explicit expression} we get
\begin{align*}
b(p) =  \sum_{r=1}^{p-1}  \mathbf{e}_{p} \left ( rN \right )\left \{ \frac{-1}{p-1} -1+  \frac{1}{p-1} \sum_{a \bmod p}  \mathbf{e}_{p} \left ( -ra \right)  \right \}. 
\end{align*}
Since $(r,p)=1$ and hence by orthogonality we have
$$
b(p)= \frac{-p}{\varphi(p) } \sum_{r=1}^{p-1}  \mathbf{e}_{p} \left ( rN \right ) = \frac{-p c_p(N)}{\varphi(p)}.
$$

Now assume $p \mid q'$ and appealing to~\eqref{rho_p^* explicit expression} we obtain
\begin{align*}
 b(p)   =  \sum_{r=1}^{p-1}  \mathbf{e}_{p} \left ( rN \right ) \left \{     (p-1)  \mathbf{e}_{p} \left ( -ra' \right ) +   \mathbf{e}_{p} \left ( -ra' \right ) - \sum_{a \bmod p}   \mathbf{e}_{p} \left ( -ra \right ) \right \}. 
\end{align*}
Since $(r,p)=1$, orthogonality gives
\begin{align*}
b(p) &  = p    \sum_{r=1}^{p-1}  \mathbf{e}_{p} \left ( r(N-a') \right )  \\
& =  p c_{p}(N-a').
\end{align*}

Next we consider $b(p^2)$. If $p^2 \nmid q'$ then we are done, otherwise expanding the Ramanujan sum and interchanging summation, we obtain
\begin{align*}
b(p^2) & = \sideset{}{^*} \sum_{r=1}^{p^2}  \mathbf{e}_{p^2} \left ( rN \right ) \sum_{a \bmod p^2 }c_{p^2}(a-a')  \mathbf{e}_{p^2} \left ( -ra \right ) \\
&  = \sideset{}{^*} \sum_{r=1}^{p^2}  \sideset{}{^*} \sum_{r_1 =1}^{p^2}  \mathbf{e}_{p^2} \left ( rN \right )  \mathbf{e}_{p^2} \left ( -ra' \right ) \sum_{a \bmod p^2}  \mathbf{e}_{p^2} \left ( a(r_1-r) \right ).
\end{align*}
Appealing to orthogonality then gives
\begin{align*}
b(p^2) & = p^2 \sideset{}{^*} \sum_{r=1}^{p^2}  \mathbf{e}_{p^2} \left ( r(N-a') \right ) \\
& = p^2 c_{p^2}(N-a').
\end{align*}
\end{proof}

Next, we derive an explicit expression for the cross product $[\vartheta_q^* | \rho_q^* ]_q$.

\begin{lemma} \label{lem:theta-rho inner product}
For all $q = q_1q_2^2$ with $\mu^2(q_1q_2)=1$, we have
$$
[\vartheta_q^* | \rho_q^* ]_q =  \mathbbm{1}_{q_2^2|q'}  \cdot  \frac{ 6t(q)  c_{q_1/(q_1,q')}(N)c_{(q_1,q')q_2^2}(N-a') }{ \pi^2 \varphi(q') \mu(q_1/(q_1,q')) \varphi(q_1/(q_1 ,q')) } .
$$
\end{lemma}

\begin{proof}
If $q_2^2 \nmid q'$ then we are done, otherwise
\begin{align*}
[\vartheta_q^* | \rho_q^* ]_q  &  = \frac{1}{q} \sum_{a \bmod q} \gamma_q^*(N-a)  \rho_q^*(a)\\
& = \frac{6t(q)}{\pi^2 q \varphi(q') } \sum_{a \bmod q} c_{q}(N-a)  \varphi(q') \rho^*_q(a) \\
& = \frac{6t(q)}{\pi^2 q \varphi(q') } \sideset{}{^*} \sum_{r=1}^{q}  \mathbf{e}_{q} \left ( rN \right ) \sum_{a \bmod q}  \varphi(q') \rho^*_q(a)  \mathbf{e}_{q} \left ( -ra \right ).
\end{align*}
Therefore by the previous lemma, we arrive at
\begin{align*}
[\vartheta_q^* | \rho_q^* ]_q 
  & = \mathbbm{1}_{q_2^2|q'} \cdot \frac{6t(q)}{\pi^2 q \varphi(q') } \prod_{ \substack{ p|q_1 \\ p \nmid q'} } \frac{-p c_p(N)}{\varphi(p)} \prod_{\substack{p|q_1 \\ p|q'}}  p c_{p}(N-a') \prod_{p|q_2}  p^2 c_{p^2}(N-a')   \\
& = \mathbbm{1}_{q_2^2|q'} \cdot \frac{6t(q) \mu(q_1/(q_1,q')) }{\pi^2 \varphi(q') }   \frac{c_{q_1/(q_1,q')}(N)}{\varphi(q_1/(q_1 ,q'))} c_{(q_1,q')q_2^2}(N-a')
\end{align*}
and the result follows.
\end{proof}

\subsection{Local models and their products}
Let $q=q_1q_2^2$ with $\mu^2(q_1q_2)=1$,
we denote
\begin{equation} \label{def:eta_q^*}
\eta_q^* =   \frac{1}{2} \left ( \frac{   \pi^2}{ 6 t(q)}    \vartheta_q^* + \frac{  \varphi(q') \varphi(q_1/(q_1,q'))}{  \mu(q_1/(q_1,q'))} \widetilde{\rho_q^*}  \right ) ,
\end{equation}
\begin{equation} \label{def:kappa_q^*}
\kappa_q^* = \frac{1}{2} \left  ( \frac{ \pi^2   }{ 6 t(q)}    \vartheta_q^* - \frac{   \varphi(q') \varphi(q_1/(q_1,q'))   }{ \mu(q_1/(q_1,q'))} \widetilde{\rho_q^*} \right ),
\end{equation}
where 
$$
\widetilde{\rho_q^*}(a) = \frac{   \mu(q_1/(q_1,q'))  c_{q_1/(q_1,q')}(a)  c_{(q_1,q')q_2^2}(a-a')}{  \varphi(q') \varphi(q_1/(q_1,q'))}
$$
is just $\rho_q^*$ but without the divisibility condition on $q'$, see Lemma~\ref{lem: rho_q^* properties}. In particular, we can write 
$$
\rho_q^* = \mathbbm{1}_{q_2^2|q'} \cdot \widetilde{\rho_q^*}.
$$

Below, we compute the  norms of $\eta_q^*$ and $\kappa_q^*$.
\begin{lemma} \label{lem: norm of eta and kappa}
We have
$$
\lVert \eta_q^* \rVert_q^2 =  \frac{1}{2} \left (   \varphi(q) +  c_{q_1/(q_1,q')}(N) c_{(q_1,q') q_2^2}(N-a') \right ),
$$
$$
\lVert \kappa_q^* \rVert_q^2 =  \frac{1}{2} \left ( \varphi(q)  -  c_{q_1/(q_1,q')}(N) c_{(q_1,q') q_2^2}(N-a') \right ),
$$
and 
$$
[\eta_q^* | \kappa_q^*]_q =[\kappa_q^*|\eta_q^* ]_q = 0.
$$
\end{lemma}

\begin{proof}
The norms for $\eta_q^*$  and $\kappa_q^*$ follows immediately from Lemma~\ref{lem: theta,rho norms},~\ref{lem:theta-rho inner product} and that both $\vartheta_q^*$ and $\rho_q^*$ are real valued.  

Showing
$$
[\eta_q^* | \kappa_q^*]_q =[\kappa_q^*|\eta_q^* ]_q = 0
$$
also follows from Lemma~\ref{lem: theta,rho norms},~\ref{lem:theta-rho inner product}.
\end{proof}

Note that  $\lVert \eta_q^* \rVert_q^2 \neq 0$ and $\lVert \kappa_q^* \rVert_q^2 = 0$ when 
$$
\frac{q_1}{(q_1,q')} \mid N \mand (q_1,q')q_2^2 \mid (N-a').
$$

 Next, we will show there is a Hermitian relationship between the \textit{global}~\eqref{eq: global product} and \textit{local}~\eqref{eq: local product} products.

Given a function $h \in \mathscr{F}( \mathbb{Z} / q\mathbb{Z} )$, we denote $\nabla_q(h)$ to be a function on $[1,N]$ defined by 
$$
\nabla_q(h)(x) = h ( x  \bmod q).
$$

 For a function $j \in \mathscr{F}([1,N])$, we denote $\Delta_q(j)$ to be a function on the positive integers such that 
$$
\Delta_q(j)(x) = q \sum_{\substack{ n \le N \\ n \equiv x [q] }} j(n).
$$

Note that $\Delta_q \nabla_q$ is not the identity. For $h_1 \in \mathscr{F}([1,N]) $ and $h_2 \in \mathscr{F}(\mathbb{Z}/q\mathbb{Z})$, we readily check that
\begin{equation} \label{eq:adjoint property}
[\Delta_q(h_1)| h_2]_q = [h_1 | \nabla_q(h_2)].
\end{equation}
Indeed
\begin{align*}
[\Delta_q(h_1)| h_2]_q  & = \sum_{a \bmod q} \sum_{\substack{n \le N \\ n \equiv a [q]}} h_1(n) \overline{h_2(a)} \\
& = \sum_{n \le N} h_1(n) \overline{h_2( n \bmod q)}  \\
&= [h_1 | \nabla_q(h_2)].
\end{align*}

Let us set $\phi_q^* =  \nabla_q \eta_q^*$ and $\psi_q^* = \nabla_q \kappa_q^*$. Now we compute the cross products of $\phi_q^*$ and $\psi_q^*$.


\begin{lemma} \label{lem: global product}
For all positive cubefree integers $m$ and $n$, we have
\begin{align*}
[\phi_m^*   | \phi_n^* ] & = N \lVert  \eta_{m}^*\rVert_{m}^2 \cdot \mathbbm{1}_{m=n} +O(  \sigma(m) \sigma(n)),  \label{eq: phi_m |  phi_n inner product bound} \\
[\psi_{m}^* | \psi_{n}^*] & =N \lVert  \kappa_{m}^*\rVert_{m}^2 \cdot  \mathbbm{1}_{m =n} +O(  \sigma(m)  \sigma(n)), \nonumber \\
[\psi_{m}^*| \phi_{n}^*] & = [\phi_{m}^* |\psi_{n}^*] = O(  \sigma(m) \sigma(n)). \nonumber
\end{align*}
\end{lemma}

\begin{proof}
Write $m = m_1 m_2^2$ and $n=n_1 n_2^2$ with $\mu^2(m_1 m_2)= \mu^2(n_1n_2)=1$. By appealing to~\eqref{def: theta_q^*},~\eqref{def:eta_q^*}, Lemma~\ref{lem: explicit gamma} and~\ref{lem: rho_q^* properties}, we expand
\begin{align*}
[\phi_{m}^*   | \phi_{n}^* ] =   \frac{1}{4} \left( S_1  +  S_2 +  S_3 + S_4 \right ),
\end{align*}
where
\begin{align*}
S_1 & =\sum_{n \le N} c_{m}(N-n) c_{n}(N-n), \\
S_2 & =  \sum_{n \le N} c_{m}(N-n) c_{n_1/(n_1,q')}(n) c_{(n_1,q')n_2^2} (n-a'), \\
S_3 & = \sum_{n \le N} c_{n}(N-n) c_{m_1/(m_1,q')}(n) c_{(m_1,q')m_2^2} (n-a'), \\
S_4 & =    \sum_{n \le N } c_{m_1/(m_1,q' )}(n) c_{(m_1,q')m_2^2}(n-a')c_{n_1/(n_1, q')}(n) c_{(n_1,q')n_2^2}(n-a').
\end{align*}

Using the explicit form of $\vartheta_q^*$ and applying Lemma~\ref{lem: ramanujan sum property}, we have
\begin{align*}
S_1 & = \sum_{n \le N} \sum_{\substack{d_1|m \\d_1|(N-n)}} \sum_{\substack{d_2|n \\d_2|(N-n)}} d_1 \mu(m/d_1 ) d_2 \mu(n/d_2 ) \\
&= \sum_{\substack{d_1|m }} \sum_{\substack{d_2|n}} d_1 \mu(m/d_1 ) d_2 \mu(n/d_2 ) \sum_{ \substack{n \le N \\ n \equiv N [ [d_1,d_2] ] }} 1 \\
& = N \sum_{\substack{d_1|m }} \sum_{\substack{d_2|n}}   \mu(m/d_1 )  \mu(n/d_2 )(d_1, d_2) +O(  \sigma(m)  \sigma(n) ),
\end{align*}
since  $[d_1, d_2](d_1, d_2)=d_1 d_2$. By Lemma~\ref{lem: double mobuis sum}, we get
$$
S_1 =  N\phi(m) \cdot  \mathbbm{1}_{m =n}+O(  \sigma(m) \sigma(n)).
$$

Next we deal with $S_4$. Expanding each for the four Ramanujan sum and taking the summation over $n$ inside, we obtain
\begin{align*}
S_4  = &  \sum_{\substack{d_1|m_1/(m_1,q') \\ d_2 |n_1/(n_1,q') }} d_1 \mu \left (\frac{m_1}{d_1(m_1,q')} \right ) d_2 \mu \left (\frac{n_1}{d_2(n_1,q')} \right ) \\
&  \times \sum_{\substack{d_3|(m_1,q') m_2^2 \\ d_4|(n_1,q')n_2^2}} d_3 \mu \left (\frac{(m_1,q') m_2^2}{d_3} \right ) d_4 \mu \left (\frac{(n_1,q') n_2^2}{d_4} \right ) \sum_{\substack{n \le N \\ n \equiv 0 [d_i], i=1,2 \\ n \equiv a' [d_j], j=3,4 }} 1.
\end{align*}
Observe that $(d_1,d_3)=(d_1,d_4)=(d_2, d_3)=(d_2,d_4)=1$ since $m_2^2 \mid q'$ and $n_2^2 \mid q'$. By the Chinese reminder theorem, the system of congruences in the summation over $n$ can be reduced to one congruence $n \equiv \sigma  [ [d_1,d_2,d_3,d_4] ]$ for some $\sigma$. Since $(d_1, d_3)=(d_2,d_4)=1$, we verify that 
$$
[d_1,d_2,d_3,d_4] = [d_1d_3, d_2d_4].
$$
Next, we glue the variables $d_1, d_3$ and $d_2,d_4$ together and get
$$
S_4=  N  \sum_{ \substack{d_1|m  \\ d_2|n}} \frac{ d_1 \mu(m /d_1) d_2 \mu(n /d_2)}{[d_1,d_2]} + O( \sigma(m)  \sigma(n) ).
$$
By Lemma~\ref{lem: double mobuis sum}, we obtain
$$
S_4 =  N \varphi(m) \cdot \mathbbm{1}_{m = n}    + O(  \sigma(m)   \sigma(n) ).
$$

Now we consider $S_2$. Again, expanding the Ramanujan sums and interchanging summations, we get
\begin{align*}
S_2 =  \sum_{d_1 | m } d_1 \mu \left ( \frac{m}{d_1} \right )  \sum_{d_2 |n_1/(n_1,q')}  &  d_2  \mu \left ( \frac{n_1}{d_2(n_1,q')} \right ) \\
& \times \sum_{d_3|(n_1,q')n_2^2} d_3 \mu \left (\frac{(n_1,q')n_2^2}{d_3} \right ) \sum_{ \substack{n \le N \\  n \equiv N [d_1] \\ n \equiv 0 [d_2] \\ n \equiv a' [d_3]}} 1.
\end{align*}
Since $(d_2,d_3)=1$, the system of congruence in the inner sum is solvable if and only if $(d_1,d_3) \mid (N-a')$ and $(d_1, d_2) \mid N$. It follows by the Chinese remainder theorem that
$$
S_2 = M +O( \sigma(m) \sigma(n)),
$$
where
\begin{align*}
M=  N \sum_{d_1 | m} d_1 \mu \left ( \frac{m}{d_1}  \right ) \sum_{ \substack{ d_2 | n_1/( n_1 ,q') \\ (d_1,d_2)|N }} & d_2 \mu   \left  ( \frac{n_1}{d_2(n_1,q')}  \right )  \\
& \times \sum_{ \substack{d_3|(n_1,q') n_2^2 \\ (d_1,d_3)|(N-a') }}  \frac{d_3 \mu((n_1,q')n_2^2/d_3)}{[d_1, d_2d_3]}.
\end{align*}
The dependency on the divisibility condition over the summation of $d_2$ and $d_3$ is troublesome. We deal with this by Lemma~\ref{lem: if d=a} and observe that
$$
\mathbbm{1}_{(d_1,d_3)|(N-a')} = \sum_{\substack{t|(N-a') \\ t|(d_1,d_3)}}  \mathbbm{1}_{t= (d_1,d_3)}  =  \sum_{\substack{t|(N-a') \\ t|(d_1,d_3)}} \sum_{\substack{k|(d_1, d_3) \\ t|k }} \mu(k/t)
$$
and
$$
\mathbbm{1}_{(d_1,d_2)|N}= \sum_{\substack{t' |N \\ t'|(d_1,d_2 )}} \mathbbm{1}_{t' = (d_1,d_2)}  = \sum_{\substack{t' |N \\ t'|(d_1,d_2 )}} \sum_{\substack{k'|(d_1, d_2) \\ t'|k' }} \mu(k'/t').
$$
Substituting this into $M$ and gluing the variables $d_2$ and $d_3$ and noticing that $(k,k')=1$ since $(d_2,d_3)=1$, $k'\mid d_2$ and $k\mid d_3$, we assert
\begin{align*}
M  =  N \sum_{\substack{t'|N \\ t'|k'|(\frac{n_1}{(n_1,q')}, m) }} & \sum_{\substack{t|(N-a') \\ t|k| ((n_1,q') n_2^2,m)}}    \mu \left (\frac{k}{t} \right ) \mu \left (\frac{k'}{t'} \right ) kk' \\
& \times \sum_{\substack{d_1 |\frac{m}{kk'} }} \mu \left (\frac{m}{d_1 kk'} \right ) \sum_{\substack{d_2| \frac{n}{kk'}  }} \mu \left (\frac{n}{d_2 kk'} \right ) (d_1,d_2). 
\end{align*}
If $m \neq n$ then $M=0$ by Lemma~\ref{lem: double mobuis sum}. If $m = n$ then again by Lemma~\ref{lem: double mobuis sum}, we get
\begin{align*}
M & = N   \sum_{\substack{t'|N \\ t'|k'| \frac{n_1}{(n_1,q')} }} \sum_{\substack{t|(N-a') \\ t|k| (n_1,q')n_2^2 }}\mu \left (\frac{k}{t} \right ) \mu \left (\frac{k'}{t'} \right ) kk' \varphi \left (\frac{n}{kk'} \right ).
\end{align*}
Notice that for $k' \mid \frac{n_1}{(n_1,q')}$ and $k \mid (n_1,q') n_2^2$, we rewrite the Euler totient function as a Dirichlet convolution
\begin{align*}
\varphi \left (\frac{n}{kk'}  \right ) & = \sum_{s | \frac{n}{kk'} } s \mu \left (\frac{n}{s kk'} \right ) \\
&  = \sum_{d_1 |\frac{n_1}{(n_1,q')k'}}  d_1 \mu \left (\frac{n_1}{k' d_1 (n_1,q')} \right ) \sum_{d_2| \frac{(n_1,q')n_2^2}{k}}  d_2 \mu \left (\frac{(n_1,q')n_2^2}{kd_2} \right ).
\end{align*}
Substituting this into the above and interchanging summation, we obtain
\begin{align*}
M  &=  N  \sum_{\substack{t'|N \\ t'|k'| \frac{n_1}{(n_1 ,q')} }}   \sum_{\substack{t|(N-a') \\ t|k| (n_1,q')n_2^2 }}\mu \left (\frac{k}{t} \right ) \mu \left (\frac{k'}{t'} \right )  \\
 &  \hspace{1cm} \times \sum_{d_1 |\frac{n_1}{(n_1,q')k'}} k' d_1 \mu \left (\frac{n_1}{k' d_1 (n_1,q')} \right ) \sum_{d_2| \frac{(n_1,q')n_2^2}{k}} k d_2 \mu \left (\frac{(n_1,q')n_2^2}{kd_2} \right ) \\
 &= N \sum_{d_1 |\frac{n_1}{(n_1,q')}} d_1 \mu \left (\frac{n_1}{d_1 (n_1,q')} \right ) \sum_{d_2|(n_1,q')n_2^2} d_2 \mu \left (\frac{(n_1,q')n_2^2}{d_2}  \right )  \\
& \hspace{5.5cm} \times \sum_{\substack{t'|N \\ t'|d_1}}  \sum_{\substack{k' |d_1 \\t'|k'}} \mu \left (\frac{k'}{t'} \right ) \sum_{\substack{t|(N-a') \\ t|d_2}} \sum_{\substack{k|d_2 \\t|k}} \mu \left ( \frac{k}{t} \right ).
\end{align*}
From Lemma~\ref{lem: ramanujan sum property} we notice
$$
\mathbbm{1}_{d_1|N} = \sum_{\substack{t'|N \\ t'|d_1}} \mathbbm{1}_{t' = d_1} = \sum_{\substack{t'|N \\ t'|d_1}} \sum_{\substack{k'|d_1 \\t'|k'}} \mu(k'/t')
$$
and
$$
\mathbbm{1}_{d_2|(N-a')} = \sum_{\substack{t|(N-a') \\ t|d_2}} \mathbbm{1}_{t = d_2} = \sum_{\substack{t|(N-a') \\ t|d_2}} \sum_{\substack{k|d_2 \\t|k}} \mu(k/t).
$$
Therefore the sum $M$ in question collapses into
\begin{align*}
M & = \sum_{\substack{d_1 |\frac{n_1}{(n_1,q')} \\ d_1 |N}} d_1 \mu \left (\frac{n_1}{d_1(n_1,q')} \right ) \sum_{\substack{d_2|(n_1,q')n_2^2 \\ d_2|(N-a')}} d_2 \mu \left (\frac{(n_1,q')n_2^2}{d_2} \right ) \\
& = c_{n_1/(n_1,q')} (N) c_{(n_1,q')n_2^2}(N-a').
\end{align*}
Hence we have
$$
S_{2} =  N c_{n_1/(n_1,q')} (N) c_{(n_1,q')n_2^2}(N-a') \cdot  \mathbbm{1}_{m=n}  +   O( \sigma(m)  \sigma(n)).
$$
Similarly we also have
$$
S_{3} =  N c_{m_1/(m_1,q')} (N) c_{(m_1,q')m_2^2}(N-a') \cdot \mathbbm{1}_{m=n}  +   O(  \sigma(m) \sigma(n)).
$$
Gathering all the estimates above, we finally get
\begin{align*}
[\phi_{m}^*    | \phi_{n}^* ]   = & \frac{\mathbbm{1}_{m=n} \cdot  N }{2} \left  (\varphi(m) +  c_{m_1/(m_1,q')} (N)  c_{(m_1,q')m_2^2}(N-a')  \right )  \\
&   +O(  \sigma(m)   \sigma(n)) 
\end{align*}
and the result follows from Lemma~\ref{lem: norm of eta and kappa}.

The remaining bounds for $[\psi_{m}^* | \psi_{n}^*], [\psi_{m}^*| \phi_{n}^*]$ and $[\phi_{m}^* |\psi_{n}^*]$ follows immediately from our computation of $S_1, S_2, S_3, S_4$.
\end{proof}

\subsection{Approximating $f$ and $g$}

In this section we impose implicitly  the condition 
$$
q' \le \mathcal{L}^{C_1}.
$$

Take our set of moduli to be
\begin{equation} \label{eqn: mathcal(Q)}
\mathcal{Q} = \left \{ q_1 q_2^2: q_1 \le Q_1, q_2 \le Q_2, \mu^2(q_1 q_2)=1 \right  \}.
\end{equation}
Here 
$$
Q_1 =  \mathcal{L}^{D_1} \mand Q_2=\mathcal{L}^{D_2},
$$
where $D_1, D_2 \ge C_1$ will be chosen later. We also set 
$$
Q =\max \{ Q_1, Q_2 \}
$$
and in particular this implies $q' \le Q \ll_{\varepsilon} N^{\varepsilon}$. The set $\mathcal{Q}$ will be referenced in the preceding lemmas.

Finally, we recall $(a',q')=1$ and let 
\begin{equation} \label{def: f}
f(n) = \Lambda(n) \cdot \mathbbm{1}_{n \equiv a' [q']}
\end{equation}
and 
\begin{equation} \label{def: g}
g(n) =    \mu^2(N-n)
\end{equation}
for  all $n \le N$.

\begin{lemma} \label{lem: inner product twist with f or g}
For any $A>0$ and $q \in \mathcal{Q}$, we have
\begin{align*}
[\phi_q^*|f] & = N[\eta_q^* | \rho_q^*]_q + O\left ( q^2 N \mathcal{L}^{-A} \right ),\\
[\phi_q^* |g] & = N[\eta_q^*|\vartheta_q^*]_q + O \left ( q^{3/2 + \varepsilon} N^{1/2} \right ),
\end{align*}
and
\begin{align*}
[\psi_q^*|f] & = -N[\kappa_q^*| \rho_q^*]_q + O \left (  q^2 N \mathcal{L}^{-A} \right ), \\
[\psi_q^* |g] &  = N[\kappa_q^*|\vartheta_q^*]_q + O \left ( q^{3/2 + \varepsilon} N^{1/2} \right ).
\end{align*}

In particular, we have
$$
|[\phi_q^* | f]| + |[\psi_q^* | f]| \ll  \mathbbm{1}_{q_2^2|q'} \cdot   N  \varphi((q_1,q')q_2^2) \varphi(q')^{-1} +   q^2 N \mathcal{L}^{-A}
$$
and
$$
|[\phi_q^* | g]| + |[\psi_q^* | g]| \ll  N  | t(q)| \varphi(q) +  q^{3/2 + \varepsilon} N^{1/2}. 
$$
All implied constant above may depend on $A, C_1 , D_1, D_2, \varepsilon$.
\end{lemma}

\begin{proof}
By~\eqref{eq:adjoint property}, we get
\begin{align*}
 [\nabla_q \eta_q^* |f] & = [\eta_q^*| \Delta_q f]_q \\
 & = [\eta_q^* | N \rho_q ]_q + [\eta_q^*| \Delta_q f - N \rho_q]_q.
\end{align*}
Note that by~\eqref{def:eta_q^*} and~\eqref{def:kappa_q^*}, we have
\begin{equation} \label{eq: rho_q}
\rho_q^* =  \frac{ \mathbbm{1}_{q_2^2|q'} \cdot  \mu(q_1/(q_1,q')) }{   \varphi(q') \varphi(q_1/(q_1,q'))  }(\eta_q^* - \kappa_q^* ) .
\end{equation}
By the M\"obius inversion formula that 
$$
\rho_q = \sum_{d|q} \rho_d^*,
$$
and it follows
\begin{align*}
[\eta_q^*|\rho_q]_q  & =  \sum_{d|q} [\eta_q^*|\rho_d^*]_q \\
& = \frac{ \mathbbm{1}_{q_2^2|q'} \cdot  \mu(q_1/(q_1,q')) }{\varphi(q') \varphi(q_1/(q_1,q'))  } \sum_{d|q} \left ( [\eta_q^*|\eta_d^*]_q -[\eta_q^*|\kappa_d^*]_q   \right ).
\end{align*}
Following the proof of Lemma~\ref{lem: global product} the summand is zero unless $d=q$, and hence $[\eta_q^*|\rho_q]_q =[\eta_q^*|\rho_d^*]_q$. It follows
\begin{align*}
[\eta_q^*  |\rho_q  ]_q  &=  \frac{\mathbbm{1}_{q_2^2|q'} \cdot    \mu(q_1/(q_1,q'))  }{ 2\phi(q') \varphi(q_1/(q_1,q')) } \left  ( \varphi(q) +  c_{q_1/(q_1,q')}(N) c_{(q_1,q')q_2^2}(N-a')  \right ) \\
& \ll  \mathbbm{1}_{q_2^2|q'}  \cdot   \varphi((q_1,q')q_2^2) \varphi(q')^{-1}.
\end{align*}

Write
$$
[\eta_q^*| \Delta_q f - N \rho_q]_q =  \frac{1}{2}( S_1 +  S_2 ),
$$
where
\begin{align*}
S_1  & = \sum_{a \bmod q}  c_q(N-a) \left (\sum_{\substack{n \le N \\ n \equiv a [q]}} f(n) - \frac{N}{q} \rho_q(a) \right ), \\
S_2  & =  \sum_{a \bmod q} c_{q_1/(q_1,q')}(a) c_{(q_1,q')q_2^2}(a-a') \left ( \sum_{\substack{n \le N \\  n \equiv a [q]}} f(n) - \frac{N}{q} \rho_{q}(a) \right ).
\end{align*}
We crudely bound using the Siegel-Walfisz Theorem (Lemma~\ref{lem: siegel-walfisz}) to get
\begin{align*}
S_1 , S_2  & \le q^2 \max_{a \bmod q} \Bigg | \sum_{\substack{n \le N \\  n \equiv a [q]}} f(n) - \frac{N}{q} \rho_{q}(a) \Bigg |\\
&  \ll_{A, C_1 ,D_1,D_2} q^2 N \mathcal{L}^{-A}.
\end{align*}

Next we turn to $[\phi_q^* |g]$. Similarly we get
\begin{align*}
[\nabla_q \eta_q^* |g]& = [\eta_q^* | \Delta _q g]_q \\
& = [\eta_q^* |N \vartheta_q]_q + [\eta_q^* | \Delta_q g -N \vartheta_q]_q.
\end{align*}
Note that by~\eqref{def:eta_q^*} and~\eqref{def:kappa_q^*}, we have
\begin{equation} \label{eq: vartheta_q}
 \vartheta_q^* =  \frac{  6   t(q)}{ \pi^2}(\eta_q^* + \kappa_q^*).
\end{equation}
By the M\"obius inversion formula, we have 
$$
\vartheta_q =  \sum_{d|q} \vartheta_q^*,
$$
and it follows
\begin{align*}
[\eta_q^* | \vartheta_q]_q  & = \sum_{d|q}[\eta_q^* | \vartheta_d^*]_q \\
&  = \frac{ 6  t(q)}{ \pi^2}\sum_{d|q}  \left ( [\eta_q^* | \eta_d^*]_q +[\eta_q^* | \kappa_d^*]_q  \right ).
\end{align*}
Again following the proof of Lemma~\ref{lem: global product}, the summand vanishes unless $d=q$, and hence $[\eta_q^* | \vartheta_q]_q =[\eta_q^* | \vartheta_d^*]_q$. Therefore
\begin{align*}
[\eta_q^* | \vartheta_q]_q & = \frac{  3  t(q) }{  \pi^2} \left (\varphi(q) +   c_{q_1/(q_1,q')} (N) c_{(q_1,q')q_2^2}(N-a') \right ) \\
 & \ll  | t(q)| \varphi(q).
\end{align*}

We can deal with $[\eta_q^* | \Delta_q g -N \vartheta_q]_q$ just like above but we apply Lemma~\ref{lem: sqfr in AP bound} instead of Lemma~\ref{lem: siegel-walfisz} to get
$$
[\eta_q^* | \Delta_q g -N \vartheta_q]_q  \ll_{\varepsilon } q^{3/2 + \varepsilon} N^{1/2}.
$$
The bounds for $[\psi_q^*|f]$ and $[\psi_q^*|g]$ follows similarly.
\end{proof}

Denote 
$$
\mathcal{E} =\left  \{ q_1q_2^2 \in \mathcal{Q}  : \frac{q_1}{(q_1,q')} \mid N , (q_1, q')q_2^2 \mid (N-a') \right  \}
$$
and note that $\lVert \kappa_q^* \rVert_q = 0$ if and only if $q \in \mathcal{E}$.
The set $\mathcal{E}$ will be referenced in the preceding lemmas.

Now we give upper and lower bounds for these norms.

\begin{lemma} \label{lem: upper and lower bound for M}
For all $q=q_1q_2^2 \in \mathcal{Q} $, we have
$$
\varphi(q)/4 \le  \lVert  \eta_q^* \rVert_q^2   \le    \varphi(q).
$$
The same holds when we replace $(\eta_q^*, \mathcal{Q})$ with $(\kappa_q^*, \mathcal{Q} \backslash \mathcal{E})$.
\end{lemma}

\begin{proof}
From Lemma~\ref{lem: norm of eta and kappa}, we get
$$
\lVert  \eta_q^* \rVert_q^2 =  \frac{1}{2} \left ( \varphi(q) +  c_{q_1/(q_1,q')}(N) c_{(q_1,q')q_2^2}(N-a') \right ) .
$$

Clearly 
$$
 |c_{q_1/(q_1,q')}(N) c_{(q_1,q')q_2^2}(N-a')  | \le \varphi(q)
$$
and the upper bound follows.

For the lower bound, note that 
$$
c_{q_1/(q_1,q')}(N) c_{(q_1,q')q_2^2}(N-a')
$$
strictly divides $\varphi(q)$, hence
$$
\varphi(q) +  c_{q_1/(q_1,q')}(N) c_{(q_1,q')q_2^2}(N-a') \ge \varphi(q) - \frac{1}{2}\varphi(q).
$$

The bound for $\lVert  \kappa_q^* \rVert_q^2$ is similar.
\end{proof}

Next we need to estimate the sums
\begin{equation} \label{eq: average over phi}
 \sum_{ t \in \mathcal{Q} } |[\phi_q^*|\phi_{t}^*]| + \sum_{ t \in \mathcal{Q} } |[\phi_q^*|\psi_{t}^*]| 
\end{equation}
for all $q \in \mathcal{Q}$ and
\begin{equation} \label{eq: average over psi}
  \sum_{ t \in \mathcal{Q} } |[\psi_q^*|\phi_{t}^*]| + \sum_{ t \in \mathcal{Q} } |[\psi_q^*|\psi_{t}^*]|.
\end{equation}
for all $q \in \mathcal{Q}$. But first, we need an a priori bound.
\begin{lemma} \label{lem: error bound for M}
We have
$$
\sum_{q \in \mathcal{Q} }   \sigma(q) \ll_{\varepsilon}  N^{\varepsilon}.
$$
\end{lemma}

\begin{proof}
The sum is majorized by
\begin{align*}
\sum_{q_1 \le  Q_1 } \sigma(q_1)   \sum_{q_2 \le Q_2} \sigma(q_2^2) & \ll_{\varepsilon} Q^{\varepsilon} \sum_{q_1 \le Q_1} q_1 \sum_{q_2 \le Q_2} q_2^2 \\
&\ll_{\varepsilon}  Q^{5+\varepsilon} \ll_{\varepsilon} N^{\varepsilon}.
\end{align*}
\end{proof}

We only show for~\eqref{eq: average over phi} as~\eqref{eq: average over psi} is similar. By Lemma~\ref{lem: global product} and~\ref{lem: error bound for M}, we get that~\eqref{eq: average over phi} is
\begin{align*}
 & N \lVert  \eta_q^* \rVert_q^2  +O \left (  \sigma(q)^2 +  \sum_{ \substack{ t \in \mathcal{Q} \\ t \neq q }} |[\phi_q^*|\phi_{t}^*]| + \sum_{ t \in \mathcal{Q} } |[\phi_q^*|\psi_{t}^*]| \right  ) \\
 & =  N \lVert  \eta_q^* \rVert^2_q + O\left (  \sigma(q) \sum_{t \in \mathcal{Q}}   \sigma(t) \right ).
\end{align*}

This motivates the following definition. Let $\varepsilon >0$ and $C=C(\varepsilon) > 0$ be large enough so that it dominates both~\eqref{eq: average over phi} and~\eqref{eq: average over psi}.
Next we set
\begin{align*}
 M(\phi_q^*)   = N \lVert  \eta_q^* \rVert^2_q + C N^{\varepsilon}
\end{align*}
for all $q \in \mathcal{Q}$ and
\begin{align*}
M(\psi^*_q) = N \lVert  \kappa_q^* \rVert^2_q  + C   N^{\varepsilon} 
\end{align*}
for all $q \in \mathcal{Q} \backslash \mathcal{E}$.

The following result shows that we can replace $M(\phi_q^*)$ and $M(\psi_q^*)$ by $N \lVert \eta_q^* \rVert_q^2$ and $N \lVert \kappa_q^* \rVert_q^2$ respectively in the summand up to an error term. 

\begin{lemma} \label{lem: switching M}
If 
\begin{equation} \label{eq: beta condition 1}
\gamma_q \ll  (  N | t(q)| \varphi(q))^2
\end{equation}
for all $q \in \mathcal{Q}$ or if
\begin{equation} \label{eq: beta condition 2}
\gamma_q \ll  N^2  | t(q)| \varphi(q) \varphi((q_1,q')q_2^2) \varphi(q')^{-1} 
\end{equation}
for all $q \in \mathcal{Q}$ then
\begin{equation} \label{eqn: switching M-same}
\sum_{q \in \mathcal{Q}  } M(\phi_q^*)^{-1} \gamma_q =\sum_{q \in \mathcal{Q}  } \frac{\gamma_q}{N \lVert \eta_q^* \rVert_q^2 } + O_{\varepsilon}(N^{\varepsilon}).
\end{equation}
The same holds if we replace $(\phi_q^*, \eta_q^*, \mathcal{Q})$ by $(\psi_q^*, \kappa_q^* , \mathcal{Q} \backslash \mathcal{E} )$.
\end{lemma}

\begin{proof}
Taking the difference and by Lemma~\ref{lem: upper and lower bound for M}, it is enough to bound
\begin{equation} \label{eqn: difference}
N^{\varepsilon} \sum_{q \in \mathcal{Q}  } \frac{\gamma_q  }{N  \varphi(q) (  N \varphi(q) +  C N^{\varepsilon})}. 
\end{equation}
First we suppose~\eqref{eq: beta condition 1} and recalling~\eqref{eq: t}, we majorize the sum above by
\begin{align*}
\sum_{q \in \mathcal{Q}  } |t(q)|^2  &  \ll \sum_{q_1 \le Q_1} |t(q_1)|^2  \sum_{\substack{q_2 \le Q_2 \\ q_2^2 |q'}}   |t(q_2^2)|^2  \\ 
& \ll_{\varepsilon} Q^{\varepsilon}  \sum_{q_1 \le Q_1} \frac{1}{q_1^4} \sum_{\substack{q_2 \le Q_2 \\ q_2^2 |q'}} \frac{1}{q_2^4} \ll_{\varepsilon}  N^{\varepsilon}.
\end{align*}
Therefore~\eqref{eqn: switching M-same} holds. The same argument holds if we replace $(\phi_q^*, \eta_q^*, \mathcal{Q})$ by $(\psi_q^*, \kappa_q^*, \mathcal{Q} \backslash \mathcal{E}  )$.

Next we assume~\eqref{eq: beta condition 2}. Using~\eqref{eqn: difference} and Lemma~\ref{lem: upper and lower bound for M}, we are lead to bound
$$
N^{\varepsilon} \sum_{q \in \mathcal{Q} } \frac{  |t(q)|  \varphi((q_1,q')q_2^2)}{\varphi(q') \varphi(q)}.
$$
Recalling~\eqref{eq: t}, the sum above is majorized by
\begin{align*}
 \frac{1}{\varphi(q')}\sum_{q_1 \le Q_1}   \frac{|t(q_1)| \varphi((q_1 ,q')) }{\varphi(q_1)} \sum_{\substack{q_2 \le Q_2 \\ q_2^2|q'}}  |t(q_2^2)|  
& \ll_{\varepsilon} Q^{\varepsilon} \sum_{q_1 \le Q_1} \frac{1}{q_1^2} \sum_{\substack{q_2 \le Q_2 }} \frac{1}{q_2^2} \\
& \ll_{\varepsilon} N^{\varepsilon}.
\end{align*}
The same holds  when we replace $(\phi_q^*, \eta_q^*, \mathcal{Q})$ by $(\psi_q^*, \kappa_q^* , \mathcal{Q} \backslash \mathcal{E} )$.
\end{proof}

\begin{lemma} \label{lem: [g|g] related bound}
For all $A,\varepsilon >0$, we have
\begin{align*}
[g|g] - \sum_{q \in \mathcal{Q} } M(\phi_q^*)^{-1} |[g|\phi_q^*]|^2  -  \sum_{\substack{q\in \mathcal{Q} \backslash \mathcal{E}  }}  M(\psi_q^*)^{-1}  &|[g|\psi_q^*]|^2
\end{align*}
is 
$$
O_{\varepsilon} \left (N^{1/2} Q^{7+\varepsilon} + NQ_1^{-2} + NQ_2^{-1} \right ).
$$
\end{lemma}

\begin{proof}

Write
$$
\sum_{q \in \mathcal{Q} } M(\phi_q^*)^{-1}|[g|\phi_q^*]|^2 = \sum_{q \in \mathcal{Q} } \beta_q [\phi_q^*|g]
$$
where $\beta_q = M(\phi_q^*)^{-1} [g|\phi_q^*]$. By Lemma~\ref{lem: inner product twist with f or g}, we replace $[\phi_q^*|g]$ by $N [\eta_q^*|\vartheta_q^*]$ up to an error term. Indeed
\begin{align*}
\sum_{q \in \mathcal{Q}} \beta_q   [\phi_q^*|g]   & =N \sum_{q \in \mathcal{Q} } \beta_q [\eta_q^*|\vartheta_q^*]_q + O_{\varepsilon} \left ( N^{1/2}  \sum_{q \in \mathcal{Q} } \beta_q  q^{3/2+\varepsilon} \right  ) \\
& = N \sum_{q \in \mathcal{Q} } \beta_q [\eta_q^*|\vartheta_q^*]_q + O_{\varepsilon} \left (N^{1/2}Q^{7+\varepsilon} \right )
\end{align*}
after recalling~\eqref{eq: t} and using the bound 
\begin{align*}
\beta_q & \ll (N|t(q)| \varphi(q) + q^{3/2+\varepsilon} N^{1/2})/(N\phi(q))\\
&  \ll 1 
\end{align*}
collected from Lemma~\ref{lem: inner product twist with f or g} and~\ref{lem: upper and lower bound for M}. Reiterating again we have
\begin{align*}
 \sum_{q \in \mathcal{Q} }   M(\phi_q^*)^{-1}  |[g|\phi_q^*]|^2  = N^2  \sum_{q \in \mathcal{Q} }  M(\phi_q^*)^{-1}|[\eta_q^* | \vartheta_q^*]_q|^2  + O_{\varepsilon} \left (N^{1/2}Q^{7+\varepsilon}  \right  ).
\end{align*}
We repeat this for the other sum and in total we get
\begin{align*}
& \sum_{q \in \mathcal{Q} } M(\phi_q^*)^{-1} |[g|\phi_q^*]|^2  + \sum_{q \in \mathcal{Q} \backslash \mathcal{E}  } M(\psi_q^*)^{-1} |[g|\psi_q^*]|^2 \\
& = N^2  \sum_{q \in \mathcal{Q} } M(\phi_q^*)^{-1}|[\eta_q^* | \vartheta_q^*]_q|^2 +  N^2\sum_{q \in \mathcal{Q} \backslash \mathcal{E} } M(\psi_q^*)^{-1} |[ \vartheta_q^* |\kappa_q^*]_q|^2 \\
& \quad  + O_{\varepsilon} \left (N^{1/2}Q^{7+\varepsilon} \right ).
\end{align*}

By Lemma~\ref{lem: switching M}, we replace $M(\phi_q^*)$ by $N \lVert \eta_q^* \rVert_q^2$ up to an error term and get
\begin{align*}
N^2  \sum_{q \in \mathcal{Q} } M(\phi_q^*)^{-1} & |[\eta_q^* | \vartheta_q^*]_q|^2   = N  \sum_{q \in \mathcal{Q} } \frac{|[\eta_q^* | \vartheta_q^*]_q|^2}{ \lVert \eta_q^* \rVert_q^2} +  O_{\varepsilon} \left (N^{1/2}Q^{7+\varepsilon} \right ).
\end{align*}

We recall from~\eqref{eq: vartheta_q}
$$
\vartheta_q^* = \frac{6t(q)}{\pi^2} (\eta_q^* + \kappa_q^*),
$$
it follows
\begin{align*}
\frac{ |[ \vartheta_q^* | \eta_q^*]_q|^2  }{ \lVert \eta_q^* \rVert_q^{2}} + \frac{|[ \vartheta_q^*| \kappa_q^*  ]_q|^2  }{ \lVert \kappa_q^* \rVert_q^{2}} & = \left ( \frac{6t(q)}{\pi^2} \right)^2\lVert \eta_q^* \rVert_q^{2} + \left( \frac{6t(q)}{\pi^2} \right )^2 \lVert \kappa_q^* \rVert_q^{2} \\
&  = \lVert \vartheta_q^* \rVert_q^2
\end{align*}
by~\eqref{eq: rho_q},~\eqref{eq: vartheta_q} and Lemma~\ref{lem: theta,rho norms},~\ref{lem: norm of eta and kappa}.

Next, we reiterate the process  and replace $M(\psi_q^*)$ by $N\lVert \kappa_q^* \rVert_q^2$ up to an error term. In total we get
\begin{align*}
& \sum_{q \in \mathcal{Q} }   M(\phi_q^*)^{-1} |[g|\phi_q^*]|^2    + \sum_{q \in \mathcal{Q} \backslash \mathcal{E}   } M(\psi_q^*)^{-1} |[g|\psi_q^*]|^2 \\
&  = N  \sum_{q \in \mathcal{Q} } \frac{|[ \vartheta_q^* | \eta_q^*]_q|^2}{ \lVert \eta_q^* \rVert_q^2} + N \sum_{q \in \mathcal{Q} \backslash \mathcal{E} } \frac{|[ \vartheta_q^*| \kappa_q^*  ]_q|^2}{ \lVert \kappa_q^* \rVert_q^2} + O_{\varepsilon}(N^{1/2}Q^{7+\varepsilon}) \\
& =  N \sum_{q \in \mathcal{Q} }\lVert \vartheta_q^* \rVert_q^2 + O_{\varepsilon} \left (N^{1/2}Q^{7+\varepsilon}  \right ).
\end{align*}

By Lemma~\ref{lem: theta,rho norms} we have 
$$
\lVert \vartheta_q^* \rVert_q^2 = \left (\frac{6t(q)}{\pi^2} \right )^2 \varphi(q)
$$
and thus completing the series we obtain
\begin{align*}
\sum_{q \in \mathcal{Q} } \lVert    \vartheta_q^*  \rVert_q^2 &  = \left (\frac{6}{\pi^2} \right )^2 \sum_{q_1 \le Q_1, q_2 \le Q_2} \mu^2(q_1 q_2 )  t(q_1)^2 t(q_2)^2 \varphi(q_1) \varphi(q_2^2)   \\
&=  \left (\frac{6}{\pi^2} \right )^2 \prod_{p=2}^{\infty} \left (1 +  \frac{p-1}{(p^2-1 )^2} + \frac{p(p-1)}{(p^2-1)^2} \right  )  +  O \left (Q_1^{-2} + Q_2^{-1} \right ) \\
& =\frac{6}{\pi^2} +O \left (Q_1^{-2} + Q_2^{-1} \right).
\end{align*}
Observe that by Lemma~\ref{lem: sqfr in AP bound}
\begin{align*}
[g|g] & = \sum_{\substack{n \le N } } \mu^2(N-n) = \frac{6}{ \pi^2} N + O \left (N^{1/2}  \right )
\end{align*}
and the result follows.
\end{proof}

\begin{lemma} \label{lem: [f|f] related bound}
The sum
\begin{align*}
[f|f] - \sum_{q \in \mathcal{Q} } M(\phi_q^*)^{-1} |[f|\phi_q^*]|^2  - \sum_{\substack{q\in \mathcal{Q} \backslash \mathcal{E} }}  M(\psi_q^*)^{-1}  &|[f|\psi_q^*]|^2 \ll N\mathcal{L}.
\end{align*}
\end{lemma}

\begin{proof}
By Lemma 1.1 and 1.2 of~\cite{R1} we see that
$$
0 \le  \sum_{q \in \mathcal{Q} } M(\phi_q^*)^{-1} |[f|\phi_q^*]|^2  + \sum_{\substack{q \in \mathcal{Q} \backslash  \mathcal{E}  }}  M(\psi_q^*)^{-1}  |[f|\psi_q^*]|^2 \le [f|f].
$$
Therefore it is enough to bound $[f|f]$ and indeed we have
$$
[f|f]= \sum_{ \substack{n \le N \\ n \equiv a' [q'] }} \Lambda(n)^2 \le \mathcal{L} \sum_{n \le N} \Lambda(n)  \ll N \mathcal{L}.
$$
\end{proof}

\begin{lemma} \label{lem: twisted f and g bound}
For all $A, \varepsilon >0$, we have
$$
\sum_{q \in \mathcal{Q} } M(\varphi_q^*)^{-1} [f|\phi_q^*][\phi_q^* | g] + \sum_{q \in \mathcal{Q} \backslash \mathcal{E}  } M(\psi_q^*)^{-1} [f|\psi_q^*][\psi_q^* | g]
$$
is 
$$
\mathfrak{S}_{a',q'}(N)N + O_{A,C_1,D_1,D_2,\varepsilon} \left ( N Q^8 \mathcal{L}^{-A}  + NQ^{\varepsilon} Q_1^{-1} \right )
$$
where $\mathfrak{S}_{a',q'}(N)$ is the singular series defined in Theorem~\ref{T1}.
\end{lemma}

\begin{proof}
For simplicity all implied constant in the $O(\cdot)$ term may depend on $A,C_1, D_1,D_2,\varepsilon$.

Set $\alpha_q = M(\phi_q^*)^{-1} [f|\phi_q^*]$. By Lemma~\ref{lem: inner product twist with f or g}, we obtain
\begin{align*}
\sum_{q \in \mathcal{Q}  } \alpha_q    [\phi_q^* | g] & = N \sum_{q \in \mathcal{Q}  } \alpha_q  [\eta_q^*|\vartheta_q^*]_q +O \left  (  N^{1/2}  \sum_{q \in \mathcal{Q}  } \alpha_q  q^{3/2+\varepsilon} \right ) \\
& = N \sum_{q \in \mathcal{Q}  } \alpha_q  [\eta_q^*|\vartheta_q^*]_q +  O \left (  N^{1/2} Q^{10+\varepsilon} \right ),
\end{align*}
by using the bound
\begin{align*}
\alpha_q  & \ll \frac{N \varphi((q_1,q')q_2^2)  \varphi(q')^{-1} +  q^2 N\mathcal{L}^{-A}}{N  \varphi(q)}  \\
& \ll_{\varepsilon} q^{1+\varepsilon}
\end{align*}
provided by Lemma~\ref{lem: inner product twist with f or g}  and~\ref{lem: upper and lower bound for M}.

Now set $\beta_q = M(\phi_q^*)^{-1}N [\eta_q^*|\vartheta_q^*]_q$. Then again by Lemma~\ref{lem: inner product twist with f or g}, we have
\begin{align*}
\sum_{q \in \mathcal{Q}  } \beta_q  [f|\phi_q^* ]  = & N \sum_{q \in \mathcal{Q}  } \beta_q [\rho_q^*|\eta_q^*]_q    + O \left ( N^{1/2} Q^{10+\varepsilon}  +  N\mathcal{L}^{-A} \sum_{q \in \mathcal{Q}  } \beta_q q^2    \right ).
\end{align*}
 By Lemma~\ref{lem: inner product twist with f or g} and~\ref{lem: upper and lower bound for M}, we assert
\begin{align*}
\beta_q & \ll \frac{N|t(q)| \varphi(q)}{N  \varphi(q)}  \\
& \le 1
\end{align*}
and recalling~\eqref{eq: t} we get that the error term is $O \left (N Q^8 \mathcal{L}^{-A} \right  )$.

We do the same for the other sum and we ultimately get
\begin{align*}
& \sum_{q \in \mathcal{Q}  } M(\phi_q^*)^{-1} [f|\phi_q^*][\phi_q^* | g] + \sum_{q \in \mathcal{Q} \backslash \mathcal{E}  } M(\psi_q^*)^{-1} [f|\psi_q^*][\psi_q^* | g] \\
& = N^2 \sum_{q \in \mathcal{Q}  } M(\phi_q^*)^{-1} [\rho_q^* | \eta_q^*] [\eta_q^*| \vartheta_q^*] + N^2 \sum_{q \in \mathcal{Q} \backslash \mathcal{E} } M(\psi_q^*)^{-1} [\rho_q^*|\kappa_q^*][\kappa_q^* | \vartheta_q^* ] \\
& \quad + O \left (N Q^8 \mathcal{L}^{-A} \right ).
\end{align*}
Next  by applying Lemma~\ref{lem: switching M} we replace $M(\phi_q^*)$ with $N \lVert \eta_q^* \rVert_q^2$ up to an error term and obtain
\begin{align*}
 N^2 \sum_{q \in \mathcal{Q}  } M(\phi_q^*)^{-1} [\rho_q^* | \eta_q^*] [\eta_q^*| \vartheta_q^*]  =N \sum_{q \in \mathcal{Q}  } \frac{[\rho_q^*|\eta_q^*] [\eta_q^* | \vartheta_q^*]}{ \lVert \eta_q^* \rVert_q^2} + O \left (   N Q^8 \mathcal{L}^{-A} \right  ).
\end{align*}

We recall from~\eqref{eq: vartheta_q} and~\eqref{eq: rho_q} that 
$$
\vartheta_q^* = \frac{6t(q)}{\pi^2} (\eta_q^* + \kappa_q^*) 
$$
and
$$
\rho_q^* =  \frac{ \mathbbm{1}_{q_2^2|q'} \cdot  \mu(q_1/(q_1,q')) }{  \varphi(q') \varphi(q_1/(q_1,q'))  }(\eta_q^* - \kappa_q^* ).
$$
It follows
\begin{align*}
\frac{[\rho_q^*|\eta_q^*]_q \cdot  [\eta_q^* | \vartheta_q^*]_q }{  \lVert \eta_q^* \rVert_q^{2}} &  + \frac{[\rho_q^*|\kappa_q^*]_q \cdot  [\kappa_q^* | \vartheta_q^*]_q }{ \lVert \kappa_q^* \rVert_q^{2}} \\
& = \mathbbm{1}_{q_2^2 |q'} \cdot  \frac{6 t(q) \mu(q_1/(q_1,q'))}{\pi^2 \varphi(q') \varphi(q_1/(q_1,q'))} \left (\lVert \eta_q^* \rVert_q^{2} - \lVert \kappa_q^* \rVert_q^{2} \right ) \\
& = [\rho_q^*|\vartheta_q^*]_q 
\end{align*}
by~\eqref{eq: rho_q},~\eqref{eq: vartheta_q} and Lemma~\ref{lem:theta-rho inner product},~\ref{lem: norm of eta and kappa}.

We reiterate the same procedure and replace $M(\psi_q^*)$ with $N \lVert \kappa_q^* \rVert_q^2$ up to an error term. In total we have
\begin{align*}
& \sum_{q \in \mathcal{Q}  }    M(\phi_q^*)^{-1}   [f|\varphi_q^*] [\varphi_q^* | g] +\sum_{q \in \mathcal{Q} \backslash \mathcal{E}   } M(\psi_q^*)^{-1} [f|\psi_q^*][\psi_q^* | g] \\
& = N \sum_{q \in \mathcal{Q}  } \frac{[\rho_q^*|\eta_q^*] [\eta_q^* | \vartheta_q^*]}{ \lVert \eta_q^* \rVert_q^2} + N \sum_{q \in \mathcal{Q} \backslash \mathcal{E}  } \frac{[\rho_q^*|\kappa_q^*] [\kappa_q^* | \vartheta_q^*]}{ \lVert \kappa_q^* \rVert_q^2}   +  O \left (N Q^8 \mathcal{L}^{-A}  \right ) \\
& = N \sum_{q \in \mathcal{Q} } [\rho_q^*|\vartheta_q^*]_q +  O \left (N Q^8 \mathcal{L}^{-A} \right )
\end{align*}

Next we compute the sum
\begin{align*}
\sum_{q \in \mathcal{Q}} & [ \rho_q^*|\vartheta_q^*  ]_q   = \frac{6}{\varphi(q') \pi^2}\sum_{q \in \mathcal{Q} } \mathbbm{1}_{q_2^2 |q'} \cdot \frac{ t(q)   c_{q_1/(q_1,q')}(N) c_{(q_1,q')q_2^2}(N-a') }{ \mu(q_1/(q_1,q')) \varphi(q_1 /(q_1,q' ))}.
\end{align*}
We complete the series and so the right hand side becomes
$$
\frac{6}{\varphi(q') \pi^2} \sum_{ \substack{q_1 ,q_2=1 \\ \mu^2(q_1 q_2)=1 }}^{\infty} \mathbbm{1}_{q_2^2 |q'} \cdot \frac{t(q_1 q_2^2)  c_{q_1/(q_1,q')}(N) c_{(q_1,q') q_2^2}(N-a')}{\mu(q_1/(q_1,q'))  \varphi(q_1/(q_1,q'))} 
$$
up to an error term of
\begin{align*}
& \ll \frac{1}{\varphi(q') } \sum_{ q_1 > Q_1} t(q_1) \varphi((q_1,q')) \sum_{ \substack{q_2 > Q_2 \\ \mu^2(q_1 q_2)=1 }} \mathbbm{1}_{q_2^2 |q'} \cdot  t(q_2^2) \varphi(q_2^2) \\
& \ll_{\varepsilon} Q^{\varepsilon}Q_1^{-1},
\end{align*}
since the number of divisors of $q'$ is at most $O_{\varepsilon}(Q^{\varepsilon})$ and by recalling~\eqref{eq: t}. The singular series can be represented as an infinite product
\begin{align*}
\prod_{p=2}^{\infty}  \left (1 + \frac{t(p)c_{p/(p,q')}(N) c_{(p,q')}( N-a')}{ \mu(p/(p,q'))  \varphi(p/(p,q'))}   + \mathbbm{1}_{p^2|q'} \cdot t(p^2) c_{p^2}(N-a')  \right ).
\end{align*}

If $p \nmid q'$ then the term in the product simplifies to 
$$
1+\frac{c_p(N)}{(p^2-1)(p-1)} .
$$
If $p \mid\!\mid q'$ then the term in the product simplifies to 
$$
 1 - \frac{c_p(N-a')}{p^2-1} .
$$
If $p^2 \mid  q'$ then the term in the product simplifies to 
$$
1 - \frac{c_p(N-a) + c_{p^2}(N-a) }{p^2-1} .
$$
The result follows.
\end{proof}

\section{Proof of Theorem~\ref{T1}}
For any $h_1, h_2 \in \mathscr{F}([1,N])$, let us denote
$$
\langle h_1|h_2 \rangle = \sum_{q \in \mathcal{Q} } M^{-1}(\phi_q^*) [h_1|\phi_q^*][\phi_q^*|h_2] + \sum_{q \in \mathcal{Q} \backslash \mathcal{E} } M^{-1}(\psi_q^*) [h_1|\psi_q^*][\psi_q^*|h_2]. 
$$

We recall $f, g$ from~\eqref{def: f},~\eqref{def: g} respectively and consequently the scalar product
$$
[f|g] = \sum_{\substack{N = n_1 + n_2 \\ n_1 \equiv a' [q']  }}  \Lambda(n_1) \mu^2(n_2)  . 
$$

By Cauchy's inequality, we assert
$$
|[f|g]-\langle f|g \rangle | \le \sqrt{([f|f]-\langle f|f \rangle) \cdot ([g|g]-\langle g|g \rangle)}.
$$
By Lemma~\ref{lem: [g|g] related bound} and~\ref{lem: [f|f] related bound}, the right hand side is majorised by
\begin{align*}
& \ll_{A, C_1, D_1, D_2, \varepsilon} \sqrt{N\mathcal{L}\left (N^{1/2} Q^{7 + \varepsilon} +NQ_1^{-2} + NQ_2^{-1} \right )}    \\
& \ll_{A,C_1, D_1, D_2, \varepsilon} N \mathcal{L}^{1/2} \left (    Q_1^{-1} + Q_2^{-1/2 }  \right  ).
\end{align*}

Hence we can approximate $[f|g]$ by $\langle f|g\rangle$ and by Lemma~\ref{lem: twisted f and g bound} we obtain
\begin{align*}
[f|g] &= \mathfrak{S}_{a',q'}(N)N   \\
& \quad + O_{A,C_1, D_1, D_2, \varepsilon} \left (  N Q^{8} \mathcal{L}^{-A} + NQ^{\varepsilon} Q_1^{-1}   + N\mathcal{L}^{1/2} \left (    Q_1^{-1} + Q_2^{-1/2} \right ) \right ).
\end{align*}
Recall from~\eqref{eqn: mathcal(Q)} that $Q_1 = \mathcal{L}^{D_1}, Q_2 = \mathcal{L}^{D_2}$, and $Q = \max \{ Q_1, Q_2 \}$.
Taking 
\begin{align*}
D_1 & = C_1 + C_2 +2, \\
D_2 &  = 2D_1,\\
 A & = 8D_2 + C_1 + C_2 + 1
\end{align*}
simplifies the error term to
$$
[f|g] = \mathfrak{S}_{a',q'}(N)N  + O_{C_1,C_2}  \left (N \mathcal{L}^{-C_1-C_2-1}  \right ).
$$
Notice that
\begin{align*}
\mathcal{R}_{a',q'}(N) & = [f|g] + \sum_{k \ge 2} \sum_{\substack{ p^k \le N \\ p^k \equiv a' [q'] \\}}  \mu^2(N-p) \log(p).
\end{align*}
The double sum can be estimated crudely by
\begin{align*}
 \sum_{k \ge 2} \sum_{\substack{ p^k \le N  }}  \log(p) & \le \sum_{2 \le k \le \log N} N^{ 1/k } \\
 &  \ll_{\varepsilon} N^{ 1/2 +\varepsilon},
\end{align*}
and consequently
$$
\mathcal{R}_{a',q'}(N) = \mathfrak{S}_{a',q'}(N)N  +  O_{C_1,C_2}  \left ( N \mathcal{L}^{-C_1-C_2-1}  \right ).
$$

If $\mathfrak{S}_{a',q'}(N)=0$ then there exists $p^2 \mid q'$ such that $p^2 \mid (N-a)$ by our remark after Theorem~\ref{T1}. Hence $\mathcal{R}_{a',q'}(N)=0$ and the result follows immediately. Otherwise if $\mathfrak{S}_{a',q'}(N) \neq 0$ then we bound from below
\begin{align*}
 \mathfrak{S}_{a',q'}(N) &  \gg \frac{1}{\varphi(q')} \prod_{ \substack{p \mid\!\mid q'  }} \left (1-\frac{1}{p+1} \right ) \\
 & \ge \frac{1}{\varphi(q')} \exp \left\{ \sum_{p \le q'} \log \left  (1- \frac{1}{p+1} \right ) \right \}.
 \end{align*}
 Applying a Taylor series expansion for the logarithm
 $$
 \log(1-x) = - \sum_{n=1}^{\infty} \frac{x^n}{n},
 $$
 valid for $|x| <1$, we obtain
 \begin{align*}
  \mathfrak{S}_{a',q'}(N) & \ge \frac{1}{\varphi(q')} \exp \left (-\sum_{p\le q'} \frac{1}{p} - \sum_{p \le q'} \sum_{n=2}^{\infty} \frac{1}{n(p+1)^n} \right )  \\
 & = \frac{1}{\varphi(q')} \exp \{ - \log \log q' +O(1) \}  \\
 & \gg_{C_1} \mathcal{L}^{-C_1-1}.
\end{align*}
Therefore 
$$
N \mathcal{L}^{-C_1-C_2-1} \ll_{C_1, C_2} \mathfrak{S}_{a',q'}(N) N \mathcal{L}^{-C_2} 
$$
and so
$$
\mathcal{R}_{a',q'}(N) =  \mathfrak{S}_{a',q'}(N) N  \left  \{1 +O_{C_1, C_2} \left  (\mathcal{L}^{-C_2} \right )   \right \}.
$$

\section*{Acknowledgement}
The author thanks I.~E.~Shparlinski and L.~Zhao for many helpful comments and discussions, and O.~Ramar\'e for helpful discussions. The author also thanks the referee for many helpful comments and suggestions. This work is supported by an Australian Government Research Training Program (RTP) Scholarship, UNSW Science PhD Writing Scholarship, and the Lift-off Fellowship of the Australian Mathematical Society.

\end{document}